\documentclass[11pt]{amsart}
\pdfoutput=1

\usepackage{verbatim} 
\usepackage{color}
\usepackage[colorlinks=true, citecolor=black, filecolor=black, linkcolor=black, urlcolor=black]{hyperref}

\newcommand\subsubsec[1]{\textbullet~\emph{#1.}~}

\def\oneleg{\Psi}
\def\Gchiral{G^+}
\def\Phiplus{\Phi^+}
\def\Phiminus{\Phi^-}

\def\bp{{\bar\partial}}
\def\bs{\bigskip}
\def\bfs{\boldsymbol}
\def\ms{\medskip}

\def\pa{\partial}
\def\sm{\setminus}
\def\ss{\smallskip}

\def\wh{\widehat}
\def\wt{\widetilde}
\def\ve{\varepsilon}

\def\eff{ \mathrm{eff}}
\def\hol{ \mathrm{hol}}
\def\Aut{ \mathrm{Aut}}

\def\id{ \mathrm{id}}
\def\SLE{ \mathrm{SLE}}
\def\Re{ \mathrm{Re}}
\def\Im{ \mathrm{Im}}

\DeclareMathOperator{\Sing}{Sing}

\def\FF{\mathcal{F}}
\def\LL{\mathcal{L}}
\def\OO{\mathcal{O}}
\def\PP{\mathcal{P}}
\def\VV{\mathcal{V}}
\def\XX{\mathcal{X}}
\def\YY{\mathcal{Y}}

\def\C{\mathbb{C}}
\def\D{\mathbb{D}}
\def\E{\mathbf{E}}
\def\H{\mathbb{H}}
\def\P{\mathbf{P}}
\def\R{\mathbb{R}}

\theoremstyle{plain}
\newtheorem*{thm*}{Theorem}
\newtheorem{thm}{Theorem}[section]
\newtheorem{lem}[thm]{Lemma}

\newtheorem{prop}[thm]{Proposition}

\theoremstyle{definition}
\newtheorem*{eg*}{Example}
\newtheorem*{egs*}{Examples}
\newtheorem*{def*}{Definition}

\theoremstyle{remark}
\newtheorem*{rmk*}{Remark}
\newtheorem*{rmks*}{Remarks}

\makeatletter
\def\@MRnumber{}
\def\@scanforMR#1 #2\endscan{%
 \def\@MRnumber{#1}%
 }
\def\reviewed#1{\ifx#1\stop \let\next=\relax \else \ifx#1(\advance\count255 by1 \else\let\next=\relax \fi \let\next=\reviewed \fi \next}
\def\MR#1{\relax
 \ifhmode\unskip\spacefactor3000 \space\fi
 \count255=0 \reviewed#1\stop
 \ifnum \count255>0
 {\@scanforMR#1\endscan\href{http://www.ams.org/mathscinet-getitem?mr=\@MRnumber}{MR#1}}
 \else
 {\href{http://www.ams.org/mathscinet-getitem?mr=#1}{MR#1}}
 \fi}
\makeatother

\newcommand\arXiv[1]{\href{http://arxiv.org/abs/#1}{arXiv:#1}}
\newcommand\arxiv[1]{\href{http://arxiv.org/abs/#1}{arXiv: #1}}

\numberwithin{equation}{section}

\begin{document}
\title{Radial SLE martingale-observables}

\author{Nam-Gyu Kang}
\address{Department of Mathematical Sciences, Seoul National University,\newline Seoul, 151-747, Republic of Korea}
\email{nkang@snu.ac.kr}
\thanks{The first author was partially supported by NRF grant 2010-0021628.
The second author was supported by NSF grant no. 1101735.}
\author{Nikolai~G.~Makarov}
\address{Department of Mathematics, California Institute of Technology,\newline Pasadena, CA 91125, USA}
\email{makarov@caltech.edu}

\date{}


\begin{abstract}
We implement a version of radial conformal field theory in a family of statistical fields generated by central charge modification of the Gaussian free field and show that the correlation functions of such fields under the insertion of one-leg operator form a collection of radial SLE martingale-observables. 
We apply the renormalization procedure to the multi-point vertex fields with the neutrality condition to expand this collection and study its basic properties. 
\end{abstract}

\maketitle

\section{Introduction and results} \label{sec: intro}

In this paper we use the method of conformal field theory to study a certain family of radial SLE martingale-observables.
To be precise, we develop a version of conformal field theory in a simply connected domain $D$ with a marked interior point $q.$ This theory is based on non-random (central charge) modification of the Gaussian free field in $D$ with the Dirichlet boundary conditions. 
It is a well-known statement in physics that under the insertion of $\oneleg(p)/(\E\,\oneleg(p))$ with $p\in\pa D$ ($\E\,\oneleg$ is the 1-point function of $\oneleg$) all correlation functions of the fields in a certain family are martingale-observables for $\SLE_\kappa(D,p\to q),$ where $\kappa$ is the parameter of the modification.  
However, $\oneleg$ here is rather vaguely specified. 
We undertake the task of defining $\oneleg$ and prove the above statement: We define the one-leg “operator” $\oneleg$ as a primary field with the desired conformal dimensions at $p$ and $q.$ 
The resulting collection of martingale-observables is further expanded by the renormalization procedure.
(We also need these operations to define $\oneleg.$)

\ms One of the goals of this paper is to outline the differences between the chordal and radial case of theory at hand. 
(The detailed account of the chordal version, i.e., the case of a simply connected domain $D$ with a marked boundary point, can be found in \cite{KM11}.) 
In spite of many similarities, the radial case has richer structure than the chordal one and is found with new phenomena: For example, the modified Gaussian free field has an additive monodromy around the marked interior point $q$ and the Virasoro field has a double pole at this point. 
Unlike in the chordal version, the neutrality condition is not automatic from conformal invariance in the radial one.  
Another goal is to apply definitions and constructions in \cite{KM11} to a different conformal setting. 
It appears that this approach can be extended to much more general settings like general Riemann surfaces with marked points, arbitrary non-random pre-pre-Schwarzian modifications, and various patterns of insertion (e.g., $N$-leg operators).
However, we certainly don't claim that what we develop here is the only relevant radial conformal field theory. 
Another relevant theory, a twisted conformal field theory is related to radial SLE in \cite{KMZ12}.

\ms \subsection{Radial SLE and martingale-observables} \label{ss: intro MO} 
Since Schramm introduced SLE in \cite{Schramm00} as the only possible candidates for the scaling limits of interface curves in critical 2-D lattice models, SLE has been used with a remarkable success to prove some important conjectures in statistical physics. 
For example, see the work of Lawler-Schramm-Werner (\cite{LSW01a}, \cite{LSW04}) and Smirnov (\cite{Smirnov01}, \cite{Smirnov10}).

\ms \subsubsec{Radial SLE} For a simply connected domain $(D,p,q)$ with a marked boundary point $p$ and a marked interior point $q,$ radial Schramm-Loewner evolution (SLE) in $(D,p,q)$ with a positive parameter $\kappa$ is the conformally invariant law on random curves from $p$ to $q$ satisfying the so-called ``domain Markov property" (see \eqref{eq: D Markov} below). 
In technical terms: For each $z \in D,$ let $g_t(z)$ be the solution (which exists up to a time $\tau_z\in(0,\infty]$) of the equation
\begin{equation} \label{eq: g}
\partial_t g_t(z) = g_t(z)\frac{\xi_t+g_t(z)}{\xi_t-g_t(z)}, \quad (\xi_t = e^{i\theta_t},\theta_t=\sqrt\kappa B_t),
\end{equation}
where $g_0:(D, p,q)\to (\D,1,0)$ is a given conformal map and $B_t$ is a standard Brownian motion with $B_0=0.$
Then for all $t,$ 
$$w_t:(D_t,\gamma_t,q)\to(\D,1,0), \qquad w_t(z):=g_t(z)/\xi_t= g_t(z)e^{-i\sqrt{\kappa}B_t}$$
is a well-defined conformal map from 
$$D_t := \{z \in D: \tau_z>t\}$$
onto the unit disc $\D.$
It is known that the SLE stopping time $\tau_z$ (defined to be the first time when the solution of Loewner equation \eqref{eq: g} does not exist) satisfies
$\lim_{t\uparrow \tau_z} w_t(z) = 1.$
The \emph{radial SLE curve} $\gamma$ is defined by the equation
$$\gamma_t \equiv \gamma(t) := \lim_{z\to1} w_t^{-1}(z)$$
and satisfies the ``domain Markov property," 
\begin{equation} \label{eq: D Markov}
\mathrm{law}\,\left( \gamma[t,\infty)\,|\,\gamma[0, t]\right)\,=\,\mathrm{law}\,\gamma_{D_t,\gamma_t,q}[0,\infty).
\end{equation}
The sets $K_t:= \{z \in \overline{\D}: \tau_z\le t\}$ are called the \emph{hulls} of the SLE.

\ms \subsubsec{Martingale-observables} Many results in the SLE theory and its applications depend on the explicit form of certain martingale-observables. 
Suppose that a non-random conformal field $M$ of $n$ variables in the unit disc is $\Aut(\D,1,0)$-invariant.
(Let us recall the definitions. 
A non-random \emph{conformal} field $f$ is an assignment of a (smooth) function
$(f\,\|\,\phi): ~\phi U\to\C$
to each local chart $\phi:U\to\phi U.$
A non-random conformal field $f$ is invariant with respect to some conformal automorphism $\tau$ of $M$ if
for all $\phi,$ $(f\,\|\,\phi)=(f\,\|\,\phi\circ \tau^{-1}).$
See Section 3.3 in \cite{KM11}.)
Conformal invariance allows us to define $M$ in any simply connected domain $(D,p,q)$ with marked points $p$ and $q,$
$$(M_{D,p,q}\,\|\,\id) = (M\,\|\,w^{-1}),$$
where $w:(D,p,q)\to(\D,1,0)$ is a conformal map.
We say that $M$ is a \emph{martingale-observable} for $\SLE_\kappa$ if for any $z_1,\cdots ,z_n\in D,$ the process
$$M_t(z_1,\cdots, z_n)=M_{D_t,\gamma_t,q}(z_1,\cdots, z_n)$$
(stopped when any $z_j$ exits $D_t$) is a local martingale on SLE probability space. 
For instance, we can use the identity chart of $D,$ and then for $[h,h_*]$-differentials $M$ with conformal dimensions $[h_q,h_{q*}]$ at $q,$ we have 
$$M_t(z) = (w_t'(z))^h (\overline{w_t'(z)})^{h_{*}} (w_t'(q))^{h_q} (\overline{w_t'(q)})^{h_{q*}}M(w_t(z)).$$

\ms \subsubsec{$\SLE_0$ observables} The case $\kappa = 0$ reveals some aspects of ``field Markov property." 
Indeed, $\SLE_0$ curves are hyperbolic geodesics, and martingale-observables (of one variable) are non-random fields $F\equiv F_{D,p,q}$ with the property 
$$F\Big|_{D_t} = F_{D_t,\gamma_t,q}.$$
In a sense one can think of them as integrals of the motion $t \mapsto \{D_t,\gamma_t,q\}$ in the corresponding Teichm\"uller space.
The reader is invited to check that 
$$\arg\big[(1-w)w^{-3/2}w'\big],\qquad S_w + \frac38 \Big(\frac{w'}w\Big)^2\Big(1-\frac{4w}{(1-w)^2}\Big)$$
are $\SLE_0$ observables. 
Here,
$$S_w = N_w' -{N_w^2}/2, \qquad N_w = (\log w)'$$
are Schwarzian and pre-Schwarzian derivatives of $w.$
(A hint: 
For the first $\SLE_0$ observable, consider the radial version of Schramm-Sheffield martingale-observables, 
$$a\Big(\arg\frac{(1-w)^2}{w} - 2\big(\frac\kappa4-1\big)\arg\frac{w'}w\Big), \qquad (a = \sqrt{2/\kappa}),$$
(see the first example in Subsection~\ref{ss: EgsMO}) and normalize them so that the limit exists as $\kappa\to0.$
For the second $\SLE_0$ observable, consider the 1-point functions of the Virasoro fields,
$$\frac{c}{12} S_w + h_{1,2} \frac{w'^2}{w(1-w)^2} + h_{0,1/2}\frac{w'^2}{w^2},$$
where the central charge $c$ and the conformal dimensions $h_{1,2}, h_{0,1/2}$ are given by 
$$c = \frac{(3\kappa-8)(6-\kappa)}{2\kappa},\quad h_{1,2} = \frac{6-\kappa}{2\kappa},\quad h_{0,1/2} = \frac{(6-\kappa)(\kappa-2)}{16\kappa}.$$
See Example \eqref{eg: T hat} in Subsection~\ref{ss: one-leg}.)

\ms \subsubsec{$\SLE_\kappa$ observables} For $\kappa >0, $ the usual way to find martingale-observables of a given conformal type is using It\^o's calculus. 
A couple of well-known important examples are referred to below. 

\ss \textbf{Example} ($\kappa = 2$). The scalar (i.e., $[0,0]$-differential)
$$M =\frac{1-|w|^2}{|1-w|^2} = \frac{P_\D(1,w)}{P_\D(1,0)} = \frac{P_D(p,z)}{P_D(p,q)}$$
played an important role in the theory of LERW, see \cite{LSW04}. 
Here $P_D$ is the Poisson kernel of a domain $D.$ 

\ss \textbf{Example} ($\kappa = 6$). The martingale-observable
$$M_t(z) = e^{t/4}\frac{\sqrt[3]{1-w_t(z)}}{\sqrt[6]{w_t(z)}}$$
is a scalar with respect to $z$ and a $[1/8,1/8]$-differential with respect to $q.$
Lawler, Schramm, and Werner applied the optional stopping theorem to the martingale $M_t(e^{i\theta})\,(0\le\theta<2\pi)$ and estimated the probability that the point $e^{i\theta}$ is not swallowed by the $\SLE_6$ hull $K_t$ at time $t$ to be 
$$\P[\, e^{i\theta} \notin K_t\,] \asymp e^{-2t \wh h_q} \sqrt[3]{\sin \frac\theta2},\qquad 2\wh h_q = 1/4,\qquad (0\le\theta<2\pi).$$
The exponent $2\wh h_q = 1/4$ is one of many exponents in \cite{LSW01c}.
See the first example (derivative exponents on the boundary) in Subsection~\ref{ss: hol 1pt field} with $\kappa=6$ and $h=0.$

\ms We will present a certain collection of radial SLE martingale-observables not by It\^o's calculus but by means of conformal field theory.

\ms \subsection{A radial CFT} \label{ss: Radial CFT} 
We consider a simply connected domain $D$ with a marked point (or a puncture) $q\in D$ and develop an $\Aut(D,q)$-invariant (rotation invariant in the $(\D,0)$-uniformization) conformal field theory. 
The \emph{Gaussian free field} $\Phi_{(0)}$ in $D$ with Dirichlet boundary condition can be viewed as a Fock space field. 
It has the $n$-point correlation function 
$$\E[\Phi_{(0)}(z_1) \cdots\Phi_{(0)}(z_n)] = \sum \prod_k 2G(z_{i_k},z_{j_k}),$$
where $G$ is the Green's function for $D$ and the sum is over all partitions of the set $\{1,\cdots,n\}$ into disjoint pairs $\{i_k,j_k\}.$
A \emph{Fock space field} is a linear combination of \emph{basic fields}, which are, by definition, formal expressions written as Wick's product ($\odot$-product) of derivatives of $\Phi_{(0)}$ and Wick's exponentials $e^{\odot\alpha\Phi_{(0)}}\,(\alpha\in\C)$ of $\Phi_{(0)}.$ 
If $X_1,\cdots,X_n$ are Fock space correlational fields and $z_1,\cdots,z_n$ are distinct points in $D,$ then
a \emph{correlation function}
$$\E[X_1(z_1)\cdots X_n(z_n)]$$
can be defined by Wick's calculus.
See Lecture 1 in \cite{KM11} for more details.
(A more traditional notation for the correlation function is $\langle X_1(z_1)\cdots X_n(z_n)\rangle.$) 
The points $z_1,\cdots,z_n$ are called the nodes of $X_1(z_1)\cdots X_n(z_n).$

\ms We define central charge modifications $\Phi\equiv\Phi_{(b)}$ of the Gaussian free field $\Phi_{(0)}$ by 
$$\Phi_{(b)} = \Phi_{(0)} + \varphi,\qquad \varphi = -2b\,\arg \frac{w'}w, \qquad (b = \sqrt{\kappa/8}-\sqrt{2/\kappa}),$$
where $w:(D,q)\to(\D,0)$ is a conformal map from $D$ onto the unit disc $\D.$
The corresponding conformal field theory is $\Aut(D,q)$-invariant and has the central charge 
$$c = 1-12b^2 = \frac{(3\kappa-8)(6-\kappa)}{2\kappa}.$$
This choice of $\varphi$ will be explained in Subsection~\ref{ss: c}. 
The non-random harmonic function $\varphi$ on $D\setminus\{q\}$ is multivalued.

\ms The \emph{bosonic} field $\Phi_{(b)}$ generates \emph{OPE family} $\FF_{(b)},$ the algebra (over $\C$) spanned by $1,\pa^j\bp^k\Phi_{(b)}$ and $\pa^j\bp^k e^{*\alpha \Phi_{(b)}}\,(\alpha\in\C)$ under the OPE multiplication.
(See Section 2.2 in \cite{KM11} for the definition of OPE product. 
If $X$ is holomorphic, then the OPE product of two fields $X$ and $Y$ is the zeroth coefficient of the regular part in the operator product expansion $X(\zeta)Y(z)$ as $\zeta\to z.$)
For example, the OPE family $\FF_{(b)}$ contains 
$$1,\quad J_{(b)}:= \pa\Phi_{(b)}, \quad J_{(b)}*J_{(b)}, \quad \pa J_{(b)}*(\Phi_{(b)}*\Phi_{(b)}) , \quad J_{(b)}*e^{*\alpha \Phi_{(b)}},\quad \textrm{etc.,}$$ 
and the \emph{Virasoro field} $T_{(b)},$
$$T\equiv T_{(b)}=-\frac12 J_{(b)}* J_{(b)}+ib\pa J_{(b)}.$$
The OPE family can be extended to multi-point fields by taking (tensor) product of 1-point fields (as long as their nodes are different).

\ms \subsubsec{Multi-vertex operators} We extend $\FF_{(b)}$ by adding the generators, the bi-variant bosonic field $\Phiplus(z,z_0)$ and (chiral) multi-vertex operators (and their derivatives)
$$\OO^{(\bfs\sigma,\bfs\sigma_*)}(\bfs z),\,\,(\bfs\sigma = (\sigma_1\cdots,\sigma_n),\,\bfs\sigma_* = (\sigma_{1*}\cdots,\sigma_{n*}),\,\bfs z = (z_1,\cdots,z_n),\,z_j\in D\sm\{q\})$$
with the \emph{neutrality condition} 
$$\sum_{j=1}^n (\sigma_j + \sigma_{j*}) = 0.$$
(Though OPE calculus is somewhat formal, fields are well-defined as Fock space fields.) 
As in the chordal case, we will define the \emph{chiral bosonic fields} $\Phiplus_{(b)}(z,z_0)$ as multivalued fields. 
(See Subsection~\ref{ss: Phi+}.) 
In Subsection~\ref{ss: O} the (chiral) \emph{multi-vertex fields} $\OO^{(\bfs\sigma,\bfs\sigma_*)}(\bfs z)$ will be defined as 
$$M^{(\bfs\sigma,\bfs\sigma_*)}(\bfs z)\, e^{\odot i\sum \sigma_j\Phiplus_{(0)}(z_j) - \sigma_{j*}\Phiminus_{(0)}(z_j)},$$
where the non-random term $M^{(\bfs\sigma,\bfs\sigma_*)}(\bfs z)$ and formal fields $\Phiplus_{(0)},\Phiminus_{(0)}$ will be explained later. 
(The chiral bosonic field $\Phiplus_{(0)}(z,z_0)$ is $\Phiplus_{(0)}(z)-\Phiplus_{(0)}(z_0).$)
The formal field $\sum \sigma_j\Phiplus_{(0)}(z_j) - \sigma_{j*}\Phiminus_{(0)}(z_j)$ is a Fock space field if and only if the neutrality condition holds. 
The vertex fields $\OO^{(\bfs\sigma,\bfs\sigma_*)}$ are $[\bfs h,\bfs h_*]$-differentials:
\begin{equation} \label{eq: dim4O} 
h_j = \frac{\sigma_j^2}2-\sigma_j b, \qquad h_{j*} = \frac{\sigma_{j*}^2}2-\sigma_{j*} b.
\end{equation}
It is natural to add chiral vertex operators $\OO^{(\bfs\sigma,\bfs\sigma_*)}$ to $\FF_{(b)}$ since Ward's OPEs hold for them: as $\zeta\to z_j\in D\sm\{q\},$
\begin{equation} \label{eq: OPE4TO}
T(\zeta)\OO(\bfs z)\sim h_{j\phantom{*}}\frac{\OO(\bfs z)}{(\zeta-z_j)^2} + \frac{\pa_{z_j}\OO(\bfs z)}{\zeta-z_j}, \quad 
T(\zeta)\bar\OO(\bfs z)\sim \bar h_{j*}\frac{\bar\OO(\bfs z)}{(\zeta-z_j)^2} + \frac{\pa_{z_j}\bar\OO(\bfs z)}{\zeta - z_j},
\end{equation}
where $T \equiv T_{(b)}$ is the Virasoro field and $\OO\equiv\OO^{(\bfs\sigma,\bfs\sigma_*)}.$
Thus the vertex fields $\OO^{(\bfs\sigma,\bfs\sigma_*)}$ are primary from algebraic point of view (as long as the neutrality condition holds). 

\ms In the next subsection we will introduce the rooted multi-vertex fields and define the extended OPE family as the algebra spanned by the fields in the OPE family and the rooted multi-vertex fields with the neutrality conditions.
As in the chordal case, under the insertion of Wick's exponential 
$$e^{\odot -\frac a2 \wt\Phi_{(0)}(p,q)}, \qquad(a = \sqrt{2/\kappa}),$$ 
($\wt\Phi_{(0)}$ is the harmonic conjugate of the Gaussian free field $\Phi_{(0)},$ i.e., $\wt\Phi_{(0)}=2\,\Im\,\Phiplus_{(0)}$), all fields in the extended OPE family $\FF_{(b)}$ of $\Phi_{(b)}$ satisfy the ``field Markov property" with respect to radial SLE filtration. 
Some forms of this main theorem appeared in physics literature, e.g., \cite{BB04} and \cite{Cardy04}. 
The preliminary version of this main theorem can be stated as follows. 

\ms For any field $X$ in the extended OPE family $\FF_{(b)}$ of $\Phi_{(b)},$ the non-random field
\begin{equation} \label{eq: main}
\E\,e^{\odot -\frac a2 \wt\Phi_{(0)}(p,q)}\, X
\end{equation}
is a martingale-observable for radial $\SLE_\kappa.$

\ms \subsection{Algebra of vertex operators} \label{ss: main} 
Let us state the main theorem for the OPE family $\FF_{(b)}.$ 
\begin{thm} \label{main X}
For the tensor product $X=X_1(z_1)\cdots X_n(z_n)$ of fields $X_j$ in the OPE family $\FF_{(b)}$ of $\Phi_{(b)},$ the non-random fields
\begin{equation} \label{eq: main X}
\E\,e^{\odot -\frac a2 \wt\Phi_{(0)}(p,q)}\, X
\end{equation}
are martingale-observables for radial $\SLE_\kappa.$ 
\end{thm}

\ms \subsubsec{Rooted multi-vertex fields} We will extend this theorem to the extended OPE family, particularly to the rooted multi-vertex fields with the neutrality condition. 
As one or several nodes of (chiral) multi-vertex fields approach the marked interior point $q,$ their correlation functions diverge. 
Normalizing multi-vertex fields and taking a limit (rooting procedure) leads to the definition of \emph{rooted vertex fields},
$$\OO^{(\bfs\sigma,\bfs\sigma_*;\tau,\tau_*)}(\bfs z) = M^{(\bfs\sigma,\bfs\sigma_*;\tau,\tau_*)}(\bfs z)\, e^{\odot i(\tau\Phiplus_{(0)}(q) - \tau_*\Phiminus_{(0)}(q)+ \sum \sigma_j\Phiplus_{(0)}(z_j) - \sigma_{j*}\Phiminus_{(0)}(z_j)}.$$
The expression of the non-random term $M^{(\bfs\sigma,\bfs\sigma_*;\tau,\tau_*)}(\bfs z)$ is not simple. 
It is given by 
$$M^{(\bfs\sigma,\bfs\sigma_*;\tau,\tau_*)}(\bfs z) = (w'(q))^{h_q}(\overline{w'(q)})^{h_{q*}} \prod_{j} \,M_j \prod_{j<k} I_{j,k},$$
where $w$ is a conformal map from $(D,q)$ to $(\D,0)$ and 
$$M_j =(w_j')^{h_j} (\overline{w_j'})^{h_{j*}}w_j^{(b+\tau)\sigma_j} \bar w_j^{(b+\tau)\sigma_{j*}} (1-|w_j|^2)^{\sigma_j\sigma_{j*}},\quad (w_j = w(z_j), w_j' = w'(z_j)). $$ 
The interaction terms $I_{jk}$ are given by 
$$I_{jk}(z_j,z_k) = (w_j-w_k)^{\sigma_j\sigma_k}(\bar w_j-\bar w_k)^{\sigma_{j*}\sigma_{k*}} (1-w_j\bar w_k)^{\sigma_j\sigma_{k*}} (1-\bar w_j w_k)^{\sigma_{j*}\sigma_k}.$$
See \eqref{eq: dim4O} for $h_j, h_{j*}$ and \eqref{eq: qdim4O} below for the conformal dimensions $h_q,h_{q*}$ at $q.$
We will explain the nature of this rooting procedure in Subsection~\ref{ss: O*}. 
(Of course, in the special case $\tau=\tau_*=0,$ the rooted vertex fields are multi-vertex fields.)

\ms Several properties are not affected by rooting procedure. 
For example, the rooted vertex fields $\OO^{(\bfs\sigma,\bfs\sigma_*;\tau,\tau_*)}$ are differentials of dimensions $[\bfs h,\bfs h_*]$ (see \eqref{eq: dim4O}) with respect to $\bfs z\in (D\sm\{q\})^n$ and Ward's OPEs~\eqref{eq: OPE4TO} hold for them.
However, the conformal dimensions $[h_q,h_{q*}]$ at $q$ are 
\begin{equation} \label{eq: qdim4O}
h_q = \frac{\tau^2}2,\qquad h_{q*} = \frac{\tau_*^2}2.
\end{equation}
Although the rooted vertex fields are defined at $\bfs z$ with $z_j \in D\sm\{q\},$ Wick's part of them contains the values of $\Phi_{(0)}^\pm$ at $q.$

\ms Applying the rooting procedure (normalizing the product of multi-vertex fields properly and then taking a limit) we arrive to the definition of the \emph{normalized tensor product} of rooted vertex fields as
$$\OO^{(\bfs\sigma_1,\bfs\sigma_{1*};\tau_1,\tau_{1*})} \star \OO^{(\bfs\sigma_2,\bfs\sigma_{2*};\tau_2,\tau_{2*})} =\OO^{(\bfs\sigma_1+\bfs\sigma_2,\bfs\sigma_{1*}+\bfs\sigma_{2*};\tau_1+\tau_2,\tau_{1*}+\tau_{2*})}.$$
We can view $\bfs\sigma$ as a divisor, a map $\bfs\sigma: D\sm\{q\}\to\R$ which takes the value $0$ at all points except for finitely many points. 
The collection of divisors forms a $\R$-linear space.
For multi-point vertex fields (as special rooted fields), their $\star$-product can be either OPE multiplication or tensor product. 
Wick's part of rooted vertex fields has sophisticated monodromy structure (which will not be studied in this paper).
Applying the rooting procedure to the bi-variant bosonic field $\Phiplus_{(b)}$, we will define the rooted bosonic field $\Phiplus_{(b),\star}.$
We extend $\FF_{(b)}$ by adding the generators, the rooted bosonic field $\Phiplus_{(b),\star},$ the rooted vertex fields with the neutrality condition, and their derivatives.
This extended collection of fields is called the \emph{extended OPE family} of $\Phi_{(b)}.$

\ms \subsubsec{One-leg operator} The so-called \emph{$N$-leg operator} is a rooted multi-vertex operator 
$$\OO^{(\bfs a, \bfs 0; -\frac N2a, -\frac N2a)}, \qquad (\bfs a = (a,\cdots,a),\bfs 0 = (0,\cdots,0)\in\R^N).$$
Denote by $\oneleg$ the one-leg operator. 
Its conformal dimensions are
$$h =h_{1,2} :=\frac{a^2}2-ab = \frac{6-\kappa}{2\kappa},\qquad h_q = h_{q*}=\frac{a^2}8=\frac1{4\kappa}, \qquad H_q:=h_q+h_{q*} = \frac1{2\kappa}$$
(recall that $a = \sqrt{2/\kappa}, b = (\kappa/4-1)a$) and 
$$\oneleg(z) = \Big(\frac{w'}{w}\Big)^h |w'_q|^{2h_q}e^{\odot ia(\Phiplus_{(0)}(z,q) + \frac12\Phi_{(0)}(q))},$$
where $w$ is a conformal map from $(D,q)$ to $(\D,0)$ and $w_q' = w'(q).$
Note that Wick's exponential $e^{\odot -\frac a2 \wt\Phi_{(0)}(p,q)}$ is not in the extended OPE family $\FF_{(b)}$ of $\Phi_{(b)}$ while the one-leg operator $\oneleg$ belongs to $\FF_{(b)}.$ 

\ms For a rooted field $\OO,$ $\E\,e^{\odot -\frac a2 \wt\Phi_{(0)}(p,q)}\, \OO$ does not make sense since both terms $\OO$ and $e^{\odot -\frac a2 \wt\Phi_{(0)}(p,q)}$ have a vertex at $q.$ 
Therefore, we need normalization. 
We can rewrite \eqref{eq: main} for $X = \OO$ as 
$$\frac{\E\,\oneleg(p,q)\star \OO}{\E\,\oneleg(p,q)}$$
and denote it by $\wh\E\,\OO.$
The main theorem holds for the rooted multi-vertex fields with the neutrality condition.

\begin{thm}\label{main O} 
For any rooted fields $\OO \equiv \OO^{(\bfs\sigma,\bfs\sigma_*;\tau,\tau_*)}$ with the neutrality condition, the non-random fields 
$$\wh\E\,\OO$$
are martingale-observables for radial $\SLE_\kappa.$ 
\end{thm}

\ms For a field $X$ in the extended OPE family $\FF_{(b)}$ of $\Phi_{(b)}$ we will introduce $\wh X$ such that 
$$\E\,\wh X = \wh\E\,X.$$ 
For example, 
\begin{gather*} 
\wh\OO^{(\sigma,\sigma_*;\tau,\tau_*)} = (w'_q)^{\wh h_q}(\overline{w'_q})^{\wh h_{q*}} (w')^ h (\overline{w'})^{ h_*} w^{(b-a/2+\tau)\sigma} \bar w^{(b-a/2+\tau_*)\sigma_*} \\
(1-w)^{a\sigma}(1-\bar w)^{a\sigma_*}(1-|w|^2)^{\sigma\sigma_*}\,
e^{\odot (i\sigma\Phiplus_{(0)}(z)-i\sigma_*\Phiminus_{(0)}(z)+i \tau\Phiplus_{(0)}(q)-i \tau_{*}\Phiminus_{(0)}(q))}, 
\end{gather*}
where the conformal dimensions $[\wh h_q,\wh h_{q*}]$ at $q$ are given by
$$\wh h_q = \frac{\tau^2}2 - \frac{\tau a}2,\qquad \wh h_{q*} = \frac{\tau_*^2}2- \frac{\tau_* a}2.$$
The 1-point field $\E\,\wh\OO$ has 9 terms. 
Thus it requires some tedious computation to verify directly that $\E\,\wh\OO$ is a martingale-observable for radial $\SLE_\kappa.$ 
We denote by $\wh\FF_{(b)}$ the image of $\FF_{(b)}$ under this correspondence $X\mapsto\wh X.$
Theorems \ref{main X} and \ref{main O} can be extended without difficulty to the non-random fields $\E \wh X,$ $\wh X \in \wh\FF_{(b)}.$

\ms \textbullet~\textbf{Remark on the one-leg operator.} 
In all physics papers the one-leg operator is specified with 
$$H_q = 2h_{0,1/2} := \frac{a^2}{4} - b^2 = \frac{(\kappa-2)(6-\kappa)}{8\kappa}$$
as in the restriction exponent (\cite{Lawler05}) and the partition function (\cite{Lawler09}). 
We can achieve this by introducing the \emph{effective one-leg operator} 
$$\oneleg^\eff := \oneleg \,\PP(q),$$
where $\PP(q)$ is the ``puncture operator" defined as a $[-b^2/2,-b^2/2]$-differential with respect to $q$ and $\PP(q) \equiv 1$ in the identity chart of $\D.$ 
(The effective one-leg operator is not a field in the extended OPE family $\FF_{(b)}$ of $\Phi_{(b)}.$ 
However, as a Fock space field, it is still in Ward's family of $\Phi_{(b)}$ i.e., $\oneleg^\eff$  and $\Phi_{(b)}$ have the same Virasoro field $T$ in $D\sm\{q\}.$)
Of course, insertion operator $\oneleg/\E\oneleg$ remains the same.

\ms \subsection{Examples of SLE martingale-observables} In Subsection~\ref{ss: EgsMO} we present several examples of radial SLE martingale-observables including Friedrich-Werner's formula. 
In the chordal case with $\kappa = 8/3$ $(c=0),$ it is well known (\cite{FW03}) that the $n$-point function $\E\,[\wh T(x_1)\cdots \wh T(x_n)\,\|\,\id_\H\,]$ of Virasoro field coincides with Friedrich-Werner's function $B(x_1,\cdots,x_n):$
$$B(x_1,\cdots,x_n) = \lim_{\ve\to0}\frac{\P(\SLE_{8/3}\textrm{ hits all }[x_j,x_j+i\ve\sqrt2])} {\ve^{2n}}.$$
We derive the radial version of this formula. 
See Theorem~\ref{radial FW}.

\ms We also use the one-leg operator $\oneleg$ to present field theoretic proof of the restriction property of radial $\SLE_{8/3}$ that for all hull $K,$
$$\P(\SLE_{8/3} \textrm{ avoids }K) = |\psi_K'(1)|^\lambda (\psi_K'(0))^\mu,\qquad (\lambda=5/8,\,\mu = 5/48),$$
where $\psi_K$ is the conformal map $(\D\sm K,0)\to (\D,0)$ satisfying $\psi_K'(0)>0.$
In particular, we explain the restriction exponents $\lambda$ and $\mu$ in terms of conformal dimensions of the one-leg operators: 
$$\lambda = h(\oneleg) := \frac{a^2}2-ab = \frac{6-\kappa}{2\kappa},\qquad \mu=H_q^\eff(\oneleg):=\frac{a^2}4-b^2 = \frac{(\kappa-2)(6-\kappa)}{8\kappa},$$
where the \emph{effective conformal dimension} $H_q^\eff(\oneleg)$ is defined by $h_q + h_{q*} - b^2$ so that $H_q^\eff(\oneleg)=H_q(\oneleg^\eff).$

\ms Examples of 1-point rooted vertex observables include Lawler-Schramm-Werner's derivative exponents (\cite{LSW01c}) of radial SLEs on the boundary: given $h$
$$\E[|w'_t(e^{i\theta})|^h \mathbf{1}_{\{\tau_{e^{i\theta}} > t\}}] \asymp e^{-2\wh h_q t} \Big(\sin^2\frac{\theta}2\Big)^{a\sigma/2} ,$$
where $\sigma$ and $\wh h_q$ are given by 
$$\sigma = \frac a 4\big(\kappa-4 + \sqrt{(\kappa-4)^2+16\kappa h } \big), \qquad \wh h_q =\frac{\sigma^2}{8} + \frac{a\sigma}4.$$

\ms\section{Radial CFT} \label{sec: Radial CFT} 

After a brief review on a stress tensor and the Virasoro field in an abstract and general setting, we introduce central charge modifications of the Gaussian free field in a simply connected domain $D$ with a marked interior point $q.$
These modifications equip the Gaussian free field with additive monodromy around $q.$ 
For example, see \cite{Dubedat09}. 
Under these modifications, the Virasoro field has a double pole at $q.$
We define Ward's functionals in terms of Lie derivatives $\LL_v$ ($v$ is a meromorphic vector field) in Subsection~\ref{ss: Wv} and then derive Ward's equations in Subsection~\ref{ss: Ward's in D}.

\ms\subsection{Stress tensor} \label{ss: W}
\ms For reader's convenience, we briefly review the definitions of a stress tensor and the Virasoro field. See Lecture 4 and 5 in \cite{KM11} for more details. Suppose $A^+$($A^-,$ respectively) is a Fock space \emph{holomorphic} (\emph{anti-holomorphic}, respectively) quadratic differential in a domain $D.$
Let $v$ be a non-random holomorphic vector field defined in some neighborhood of $p(\in D).$
We define the residue operator $A_v^+$($A_v^-$, respectively) as an operator on Fock space fields:
$$
(A_v^+X)(z)= \frac1{2\pi i}\oint_{(z)} vA^+\,X(z), \qquad
(A_v^-X)(z)= -\frac1{2\pi i}\oint_{(z)} \bar vA^-\,X(z)
$$
in a given chart $\phi,\;\phi(p)=z.$
A pair $$W=(A^+,A^-)$$ of a holomorphic quadratic differential $A^+$ and an anti-holomorphic quadratic differential $A^-$ is called a \emph{stress tensor} for a Fock space field $X$ in $D$ if for all non-random local vector fields $v,$ the so-called ``residue form of Ward's identity"
\begin{equation} \label{eq: LvWv}
\LL_vX= A^+_vX+ A^-_vX
\end{equation}
holds in the maximal open set $D_\hol (v)$ where $v$ is holomorphic. 
We recall the definition of Lie derivatives (see Section 3.4 in \cite{KM11}):
$$(\LL_v X\,\|\, \phi) = \frac d{dt}\Big|_{t=0} (X\,\|\, \phi\circ\psi_{-t}),$$
where $\psi_t$ is a local flow of $v,$ and $\phi$ is an arbitrary chart. 
We define the $\C$-linear part $\LL_v^+$ and anti-linear part $\LL_v^-$ of the Lie derivative $\LL_v$ by 
$$2\LL_v^+ = \LL_v-i\LL_{iv},\qquad 2\LL_v^- = \LL_v+i\LL_{iv}.$$
We denote by $\FF(W)\equiv\FF(A^+,A^-)$ \emph{Ward's family} of $W,$ the linear space of all Fock space fields $X$ with a common stress tensor $W.$

\ms Let us recall some cases of the transformation laws. 
A non-random conformal field $f$ is a \emph{differential} of conformal dimensions (or degrees) $[\lambda,\lambda_*]$ if for any two overlapping charts $\phi$ and $\wt\phi,$ we have
$$f = (h')^\lambda(\overline{h'})^{\lambda_*}\wt{f}\circ h,$$
where $h=\wt\phi\circ\phi^{-1}:~ \phi(U\cap\wt U)\to \wt\phi(U\cap\wt U)$ is the transition map, and $f$ (resp. $\wt f$ ) is the notation for $(f\,\|\,\phi),$ (resp. $(f\,\|\,\wt\phi)$). 
\emph{Pre-pre-Schwarzian forms}, \emph{pre-Schwarzian forms}, and \emph{Schwarzian forms} of order $\mu(\in\C)$ are fields with transformation laws 
$$f=\wt{f}\circ h +\mu \log h',\qquad f=h'\wt{f}\circ h +\mu N_h, \qquad f=(h')^2\wt{f}\circ h + \mu S_h,$$
respectively, where 
$$N_h =(\log h')',\qquad S_h = N_h' -{N_h^2}/2$$
are pre-Schwarzian and Schwarzian derivatives of $h.$ 
Transformation laws can be extended to the random conformal fields. 
See Section 3.2 in \cite{KM11}.

\ms A Fock space field $T$ is called to be the \emph{Virasoro field} for Ward's family $\FF(A,\bar A)$ if
\renewcommand{\theenumi}{\alph{enumi}}
{\setlength{\leftmargini}{1.7em}
\begin{enumerate}
\ss\item $T\in\FF(A,\bar A),$ and
\ss\item $T-A$ is a non-random holomorphic Schwarzian form.
\end{enumerate}}

\ms\subsection{Central charge modification} \label{ss: c}

 For a simply connected domain $D$ with a marked interior point $q\in D,$ and a conformal map $$w\equiv w_{D,q}:~(D,q)\to(\D,0),$$
from $D$ onto the unit disc $\D = \{z\in\C:|z|<1\},$
we consider
$$\Phi\equiv \Phi_{(b)}=\Phi_{(0)} + \varphi,$$
where $\Phi_{(0)}$ is the Gaussian free field in $D$ and $\varphi$ is a (multivalued) harmonic function in $D\sm\{q\}.$
We require that $\varphi$ is a real part or an imaginary part of a pre-pre-Schwarzian form of order $ib$ for some fixed parameter $b.$ 
This parameter $b$ is related to the central charge $c$ in the following way:
$$c = 1 - 12b^2.$$
We also require that $\varphi$ does not depend on the choice of the conformal map.
Thus we consider
$$\varphi = -2b\arg \frac{w'}w, \qquad (b\in\R),$$ 
or
$$\varphi = 2ib \log|w'|, \,\, 2ib\log \Big|\frac{w'}w\Big|, \qquad (ib\in\R\setminus\{0\}).$$
In the latter case, the parameter $b$ corresponds to $c>1.$
However, throughout this paper, we only consider conformal field theory with $c\le1,$ which has connection with $\SLE_\kappa.$ In this case, we have 
$$c = 1-\frac32\frac{(\kappa-4)^2}\kappa= \frac{(3\kappa-8)(6-\kappa)}{2\kappa}.$$

\ms For a real parameter $b,$ we now define 
\begin{equation} \label{eq: GFFb}
\Phi\equiv \Phi_{(b)}=\Phi_{(0)} + \varphi,\qquad \varphi=-2b\arg \frac{w'}w.
\end{equation}
Since the multivalued function $\varphi$ does not depend on the choice of the conformal map, the ``bosonic" field, $\Phi_{(b)}$ in $D\setminus\{q\}$ is invariant with respect to $\Aut(D,q).$ 
Also we define 
\begin{equation} \label{eq: Jb}
J\equiv J_{(b)}=\pa \Phi=J_{(0)}+j,\qquad j=\pa\varphi=ib\left(\frac{w''}{w'}-\frac{w'}{w}\right).
\end{equation}
\ms We denote by $\FF_{(b)}$ the OPE family of $\Phi_{(b)}.$ 
(Recall that $\FF_{(b)}$ is the algebra (over $\C$) spanned by $1,\pa^j\bp^k\Phi_{(b)}$ and $\pa^j\bp^k e^{*\alpha \Phi_{(b)}}\,(\alpha\in\C)$ under the OPE multiplication.)
The fields in $\FF_{(b)}$ are invariant with respect to $\Aut(D,q)$ since the OPE coefficients of two conformally invariant fields are conformally invariant.
\begin{rmks*}
{\setlength{\leftmargini}{1.7em}
\begin{enumerate}
\renewcommand{\theenumi}{\alph{enumi}}
\item The 1-point function $\varphi$ has monodromy $4\pi b$ around $q.$
\smallskip\item The 1-point function $j$ has a simple pole at $q,$ and $J$ is holomorphic in $D\setminus\{q\}.$
\smallskip\item As a pre-Schwarzian form (of order $ib$), the current $J$ is conformally invariant with respect to $\Aut(D,q).$
\end{enumerate}
}
\end{rmks*}

\ms The Virasoro field $T_{(b)}$ for $\FF_{(b)}$ in the radial case is modified in the same way (see Proposition~\ref{Tb} below) as in the chordal case. 
However, it has different properties in the radial case: 
\renewcommand{\theenumi}{\alph{enumi}}
{\setlength{\leftmargini}{1.7em}
\begin{enumerate}
\ss\item
its 1-point function is not trivial even in the disc uniformization
$$\E\, T_{(b)}(z) = -\frac{b^2}2\frac{1}{z^2};$$
\item it has a double pole at $q;$
\ss\item with the choice of $A_{(b)}$ such that $\E\, A_{(b)} = 0$
\begin{equation} \label{eq: TA}
T_{(b)} = A_{(b)} + \frac c{12} S_w -\frac{b^2}2 \Big(\frac{w'}w\Big)^2,
\end{equation}
where 
$c=c(b)=1-12 b^2$ is the central charge of $\Phi_{(b)}.$ 
\end{enumerate}}
 
\begin{prop} \label{Tb}
The field $\Phi_{(b)}$ has a stress tensor in $D\sm\{q\},$ and the Virasoro field for $\FF_{(b)}$ is
\begin{equation} \label{eq: Tb}
T\equiv T_{(b)}=-\frac12 J* J+ib\pa J.
\end{equation}
\end{prop}

\begin{proof} As in the chordal case, computation is straightforward since $\Phi_{(b)}$ is a real part of pre-pre-Schwarzian form. 
We define
\begin{equation} \label{eq: Ab}
A\equiv A_{(b)}=A_{(0)}+(ib\pa -j)J_{(0)}.
\end{equation}
Then $A$ is a holomorphic quadratic differential in $D\sm\{q\}$ with a simple pole at $q.$
As in the chordal case, $(ib\pa -j)J_{(0)}$ is a quadratic differential.

\ms We claim that $\Phi$ has a stress tensor $W=(A,\bar A).$
Since $\Phi$ is a real part of pre-pre-Schwarzian form, we just need to check Ward's OPE in $\D\setminus\{0\}$, that is,
\begin{equation} \label{eq: OPE(A,Phi)}
A(\zeta) \Phi(z) \sim \frac{J(z)}{\zeta-z}+ib\frac{1}{(\zeta-z)^2}.
\end{equation}
However, it follows directly from the following operator product expansions:
$$A_{(0)}(\zeta)\Phi_{(0)}(z)\sim \frac{J_{(0)}(z)}{\zeta-z},\qquad \pa J_{(0)}(\zeta)\Phi_{(0)}(z)\sim \frac{1}{(\zeta-z)^2},$$
and
$$j(\zeta)J_{(0)}(\zeta)\Phi_{(0)}(z)\sim\frac{ib}{z(\zeta-z)}.$$
\ms Finally, we claim that $T$ is the Virasoro field for $\FF_{(b)}.$
It is clear that $T$ has a stress tensor $W$ since it is a stress tensor for all fields in the OPE family of $\Phi.$
Also, $T$ is a Schwarzian form of order $c/12$.
Indeed,
\begin{equation} \label{eq: TA2}
T = A + \frac{1-12b^2}{12} S_w-\frac{b^2}{2}\left(\frac{w'}{w}\right)^2,
\end{equation}
where $S_w = (w''/w')'- (w''/w')^2/2$ is Schwarzian derivative of $w.$
The equation \eqref{eq: TA2} follows from \eqref{eq: Jb}, \eqref{eq: Tb}, and \eqref{eq: Ab}.
\end{proof}

\ms \subsection{Ward's functionals} \label{ss: Wv}
Let us recall the definition of Ward's functional.
(See Sections 4.5 -- 4.6 in \cite{KM11}.)
For an open set $U$ such that $\bar U \subset \bar D\sm\{q\}$ and a smooth vector field $v$ in $\bar U,$ we define 
$$W^+(v;U)=\frac1{2\pi i}\int_{\pa U} vA-\frac1{\pi }\int_{U} (\bp v)A,$$
$W^-(v;U) = \overline{W^+(v;U)},$ and $W(v;U) = 2\,\Re\,W^+(v;U).$
Then the correlations of $W^+(v;U)$ with Fock space functionals $\XX$ in $\bar D\sm\{q\}$ are well-defined provided that the nodes of $\XX$ are in $D_\hol \cap U.$ 
Recall that $D_\hol(v)$ is the maximal open set where $v$ is holomorphic.
By Green's formula we can symbolically write $W^+(v;U)$ as a correlation functional
$$W^+(v;U)=\frac1{\pi }\int_{U} v(\bp A),$$
so that we can replace $A$ with the Virasoro field $T.$
Recall that the following statements are equivalent:
\renewcommand{\theenumi}{\alph{enumi}}
{\setlength{\leftmargini}{1.7em}
\begin{enumerate} 
\ss\item The residue form of Ward's identity~\eqref{eq: LvWv} holds for $X$ in $U \cap D_\hol(v).$
\ss\item Ward's OPE holds for $X,$ i.e., for $z\in U \cap D_\hol(v),$
$$\Sing_{\zeta\to z}[A(\zeta)X(z)]= (\LL_{k_\zeta}^+X)(z), \quad \Sing_{\zeta\to z}[A(\zeta)\bar X(z)]= (\LL_{k_\zeta}^+\bar X)(z),$$
where $\Sing_{\zeta\to z}[A(\zeta)X(z)]$ is the singular part of the operator product expansion $A(\zeta)X(z)$ as $\zeta\to z$ and $k_\zeta$ is the (local) vector field defined by $(k_\zeta\,\|\,\phi)(\eta) = 1/(\zeta-\eta)$ for a given chart $\phi:U\to\phi U.$
\ss\item For $z\in U \cap D_\hol(v),$
$$\E\, \YY\,\LL(v,U) X(z) =\E\,W(v;U)X(z) \YY$$
for all correlation functionals $\YY$ with nodes in $(D\sm \bar U)\sm\{q\}.$
\end{enumerate}}

\begin{rmk*}
In the case that $\phi$ is a chart in an open set $U,$ we write $\LL(v,U)$ for the Lie derivative in $U.$ 
In the identity chart of $\D\sm\{0\},$ we simply write $\LL_v$ for $\LL(v,\D\sm\{0\}).$ 
We sometimes write $W_v^+$ for $W^+(v;D\sm\{q\})$ (see its definition below).
\end{rmk*}

\ms Let $v$ be a meromorphic vector field in $D$ continuous up to the boundary, and let $\{p_j\}$ be the poles of $v.$
We define Ward's functional $W^+(v;D\sm\{q\})$ by 
$$W^+(v;D\sm\{q\}) = \lim_{\ve\to0} W^+(v;U_\ve),$$
where $U_\ve = D\sm B(q,\ve) \sm\bigcup B(p_j,\ve).$ 
Thus we have 
\begin{equation} \label{eq: Wplus}
W^+(v;D\sm\{q\})=\frac1{2\pi i}\int_{\pa D} vA-\frac1{\pi }\int_{D} (\bp v)A,
\end{equation}
where $\bp v$ is interpreted as a distribution.
Somewhat symbolically, we have 
$$W^+(v;D\sm\{q\})=\frac1{2\pi i}\int_{\pa D} vA-\sum_j\frac1{2\pi i}\oint_{(p_j)} vA-\frac1{2\pi i}\oint_{(q)} vA.$$
For fields $X_j\in\FF(W)$ and $z_j\in D_\hol(v),$ Ward's identity 
\begin{equation} \label{eq: Ward in Dq}
\E\,W^+(v;D\sm\{q\})\, X_1(z_1)\cdots X_n(z_n) = \E\,\LL^+(v,D\sm\{q\})\,(X_1(z_1)\cdots X_n(z_n))
\end{equation}
holds.

\ms We now define Ward's functional $W^+(v;D)$ in $D$ in terms of $W^+(v;D\sm\{q\})$ and the residue operator $T_v$ at $q:$ 
$$W^+(v;D) = W^+(v;D\sm\{q\}) + T_v(q),\qquad T_v(z) = \frac1{2\pi i}\oint_{(z)} vT.$$
We also define $W^-(v;D) = \overline{W^+(v;D)},$ and $W(v;D) = 2\,\Re\,W^+(v;D).$

\ms Given a meromorphic vector field $v$ with a local flow $z_t$ in $\wh\C,$
we define its reflected vector field $v^\#$ with respect to the unit circle $\partial \mathbb{D}$ by the vector field of the reflected flow $z_t^\#$ of $z_t.$
Indeed, if $z_t$ is the flow of $v$, then $z_t^\# = 1/\bar z_t$ is the reflected flow and
its vector field $v^\#$ is given by the equation
$$\dot{z}^\# = -\overline{\dot{z}/z^2} = -\overline{v(z)}(z^\#)^2.$$
Thus
$$v^\#(z) = -\overline{v\left(\frac{1}{\bar z}\right)} z^2.$$
We now represent the quadratic differential $A$ in terms of Ward's functionals with the Loewner vector field $v_\zeta.$

\ms \begin{prop} \label{represent A}
Suppose $A$ is a holomorphic quadratic differential in $\D\setminus\{0\}$ and $W=(A,\bar A).$
If $A$ is continuous and real on the boundary (in all standard boundary chart), then
\begin{equation} \label{eq: A}
(A\,\|\,\id_{\D\sm\{0\}})(\zeta)=\frac1{2\zeta^2}\left(W^+\left(v_\zeta;\D\sm\{0\}\right)+ W^-\left(v_{\zeta^*};\D\sm\{0\}\right)\right),
\end{equation}
where the Loewner vector field $v_\zeta$ is given by
$$(v_\zeta\,\|\,\id)(z) = z \frac{\zeta+z}{\zeta-z},$$
and $\zeta^*:=1/\bar{\zeta}$ is the symmetric point of $\zeta$ with respect to the unit circle.
\end{prop}

\begin{proof}
For $\zeta\in\D,$ the reflected vector field of $v_\zeta$ has the same form.
Indeed, $v_\zeta^\#=v_{\zeta^*}.$
The vector field $v_\zeta$ is meromorphic in $\wh\C$ without poles on $\pa\D$ and its
reflected vector field is holomorphic in $\D.$
We write $W^+(v_\zeta)$ for $W^+\left(v_\zeta;\D\sm\{0\}\right).$
It follows from \eqref{eq: Wplus} that 
$$W^+(v_\zeta)=-\frac1\pi\int_\D A~\bp v_\zeta+\frac 1{2\pi i}\int_{\pa \D} v_\zeta A,\qquad W^+(v_\zeta^{\#})=\frac1{2\pi i}\int_{\pa \D} v_\zeta^{\#}A.$$
Since $A$ is real on the boundary, we have
$$\frac 1{2\pi i}\int_{\pa \D} v_\zeta A=-\overline{ W^+(v_\zeta^{\#})},
\qquad
\frac1\pi\int_\D A~\bp v_\zeta=-W^+(v_\zeta)-\overline{ W^+(v_\zeta^{\#})}.$$
However, $\bar\partial v_\zeta=-2\pi \zeta^2\delta_\zeta.$
Thus
$$W^+(v_\zeta)+\overline{ W^+(v_\zeta^{\#})} = -\frac{1}{\pi} \int_{\mathbb{\D}}A \bar\partial v_\zeta = 2\zeta^2A(\zeta),$$
which completes the proof.
\end{proof}

A quadratic differential $A_{(b)}=A_{(0)}+(ib\pa -j)J_{(0)}$ in the proof of Proposition~\ref{Tb} satisfies the conditions of previous proposition.
This field $A_{(b)}$ is real on the boundary in all standard boundary charts. 
For example, in $(\H,i)$-uniformization, 
$$\varphi =2b\arg(1+z^2),\qquad j = -2ib\frac{z}{1+z^2},$$
and the field $J_{(0)}$ is purely imaginary on the boundary.
The Virasoro field $T_{(b)}$ has a double pole at $q.$ 
This pole structure affects the representation of $T_{(b)}$ in the disc uniformization. 

\ms\begin{lem} \label{Tv}
For a meromorphic vector field $v$ with a zero at $q$ and a Fock space correlation functional $\XX$ in $D\sm\{q\},$
$$\frac1{2\pi i}\,\E\oint_{(q)}vT_{(b)}\,\XX = -\frac{b^2}2v'(q)\E\,\XX.$$
\end{lem}

(Here, $v'(q)$ is a number. In particular, $v'(q)=1$ for the Loewner vector field $v = v_\zeta.$)

\begin{proof}
A quadratic differential $A_{(b)}=A_{(0)}+(ib\pa -j)J_{(0)}$ has a simple pole at $q.$ 
Thus 
$$\frac1{2\pi i}\,\E\oint_{(q)}vT_{(b)}\,\XX = -\frac{b^2}2 \frac1{2\pi i}\,\E\oint_{(q)}\Big(\frac{w'}w\Big)^2v\,\XX= -\frac{b^2}2v'(q)\E\,\XX.$$
\end{proof}

\begin{rmk*}
If $v'(q) = 1,$ then  in the sense of correlations with Fock space correlation functionals in $D\sm\{q\},$
$$W(v;D) = \LL(v,D\sm\{q\}) -b^2I = \LL(v,D)\PP(q),$$
where $\PP(q)$ is the ``puncture operator" defined as a $[-b^2/2,-b^2/2]$-differential with respect to $q$ and $\PP(q) \equiv 1$ in the identity chart of $\D.$ See Remark in Subsection~\ref{ss: main}. 
\end{rmk*}

\ms Later, we allow Fock space correlation functionals to have a node at $q,$ e.g., the evaluations of rooted vertex fields. 
The following proposition is immediate from the previous proposition and lemma.

\ms \begin{prop} \label{represent T} In the sense of correlations with Fock space correlation functionals whose nodes are in $\overline\D\sm\{\zeta,q\},$
\begin{equation} \label{eq: T}
(T_{(b)}\,\|\,\id)(\zeta)=\frac1{2\zeta^2}\left(W_{(b)}^+\left(v_\zeta;\D\right)+ W_{(b)}^-\left(v_{\zeta^*};\D\right)\right).
\end{equation}
\end{prop}

\ms \subsection{Ward's equations in the disc} \label{ss: Ward's in D}

Ward's identities state that the action of Lie derivative operator $\LL_v^+$ can be represented by the insertion of Ward's functional $W^+(v)$ into the correlations of fields in the OPE family of $\Phi_{(b)}.$
See \eqref{eq: Ward in Dq} or Section 4.5 in \cite{KM11} for more details. 
Combining Ward's identity with the representation of $A$ (Proposition~\ref{represent A}), we derive the useful equation in the disc uniformization. 

\ms Let us express $\LL_{v_\zeta}^+(z)Y(z)$ in terms of the singular part of operator product expansion of $A(\zeta)$ and $Y(z)\,(Y\in\FF_{(b)}).$ 
Suppose that 
$$A(\eta)Y(z)\sim \sum_{j\le-1}{C_j(z)}{(\eta-z)^j}, \qquad (\eta\to z).$$
We write $A*_nY$ for $C_n$ and $\Sing_{\zeta\to z} [A(\zeta)Y(z)]$ for $\sum_{j\le-1}{C_j(z)}{(\zeta-z)^j}.$
Then
$$\LL_{v_\zeta}^+(z)Y(z) = \frac1{2\pi i}\oint_{(z)}v_\zeta(\eta)A(\eta)Y(z)\,d\eta = \sum_{j\le-1}{C_j(z)} \frac1{2\pi i}\oint_{(z)}(\eta-z)^j\,v_\zeta(\eta)\,d\eta.$$
Using 
$$\frac1{2\pi i}\oint_{(z)}(\eta-z)^j\,v_\zeta(\eta)\,d\eta =
\begin{cases}
{2\zeta^2}{(\zeta-z)^j} - (2\zeta + z) &\quad\textrm{if } j=-1;\vspace{.3em}\\
{2\zeta^2}{(\zeta-z)^j} - 1 &\quad\textrm{if } j=-2;\vspace{.3em}\\
{2\zeta^2}{(\zeta-z)^j} &\quad\textrm{if } j\le-3,
\end{cases}
$$
we get
\begin{equation} \label{eq: sing OPE}
\LL_{v_\zeta}^+(z)Y(z) = 2\zeta^2 \Sing_{\zeta\to z} [A(\zeta)Y(z)] - (2\zeta+z)A*_{-1}Y(z) - A*_{-2}Y(z).
\end{equation}

\ss \begin{prop} 
We assume that $A$ satisfies the conditions of Proposition~\ref{represent A}.
Let $Y, X_1,\cdots, X_n\in \FF(W)$ and let $X$ be the tensor product of $X_j$'s.
Then
\begin{align*}
\E\,&Y(z)\LL_{v_z}^+X + \E\LL_{v_{z^*}}^-(Y(z)X)\\
&= 2z^2\,\E[(A*Y)(z)X] +3z\,\E[(A*_{-1}Y)(z)X] + \E[(A*_{-2}Y)(z)X], 
\end{align*}
where all fields are evaluated in the identity chart of $\D\sm\{0\},$ and $\LL_v^\pm = \LL^\pm(v,\D\sm\{0\}).$
\end{prop}

\begin{proof}
Subtracting the singular part of OPE, it follows from \eqref{eq: sing OPE} that 
$$(A\ast Y)(z)=\lim_{\zeta\to z}\Big[A(\zeta)Y(z)-\frac1{2\zeta^2}\big((\LL^+_{v_\zeta}Y)(z)+(2\zeta+z)(A*_{-1}Y)(z)+(A*_{-2}Y)(z)\big)\Big].$$
We write $W^+(v_\zeta)$ for $W^+\left(v_\zeta;\D\sm\{0\}\right).$ 
By Proposition \ref{represent A} and Leibniz's rule, 
\begin{align*}
\E\,[A(\zeta) Y(z)\,X]&=\frac1{2\zeta^2}\big(\E\,[W^+(v_\zeta) \,Y(z)\,X]+\E\,[W^-(v_{\zeta^*}) \,Y(z)\,X]\big)\\
&=\frac1{2\zeta^2}\big(\E\,Y(z) \LL^+_{v_\zeta}X+\E (\LL^+_{v_\zeta} Y)(z)X+ \E\,\LL^-_{v_{\zeta^*}} [Y(z)X]\big).
\end{align*}
Thus
$$\frac1{2\zeta^2}\big(\E\, [Y(z)\,\LL^+_{v_\zeta}\,X]+ \E \,\LL^-_{v_{\zeta^*}}[Y(z)X] - (2\zeta+z)(A*_{-1}Y)(z)-(A*_{-2}Y)(z)\big)$$
tends to $\E\,[(A\ast Y)(z)\,X]$ as $\zeta\to z.$
\end{proof}

The following proposition follows immediately since a quadratic differential $A_{(b)}=A_{(0)}+(ib\pa -j)J_{(0)}$ satisfies the conditions of Proposition~\ref{represent A}.
It is the version of Ward's equation that we will use in the proof of Theorem~\ref{main X}. 

\ms \begin{prop} \label{T-Ward}
Let $Y, X_1,\cdots, X_n\in \FF_{(b)}$ and let $X$ be the tensor product of $X_j$'s.
Then
\begin{align*}
\E\,&Y(z)\LL_{v_z}^+X + \E\,\LL_{v_{z^*}}^-(Y(z)X)\\ 
&=2z^2\,\E[(L_{-2}Y)(z)X] 
+3z\,\E[(L_{-1}Y)(z)X] +\E[(L_0Y)(z)X] + b^2\,\E[Y(z)X] , 
\end{align*}
where all fields are evaluated in the identity chart of $\D\sm\{0\},$ and $\LL_v^\pm = \LL^\pm(v,\D\sm\{0\}).$
\end{prop}

Here, $L_n$'s are the modes of the Virasoro field $T\equiv T_{(b)}:$ 
$$L_n(z):=\frac1{2\pi i}\oint_{(z)}(\zeta-z)^{n+1} \,T(\zeta)~d\zeta.$$

\ms\section{(Multi-)Vertex fields} \label{sec: O}
From Subsection~\ref{ss: VV} to~\ref{ss: V}, we briefly discuss some basic properties of non-chiral vertex fields, chiral bosonic fields, and chiral vertex fields. 
As we mentioned in the introduction, we expand the OPE family of $\Phi_{(b)}$ in such way that it contains the rooted multi-vertex fields with the neutrality condition. 
This extended OPE family forms a subcollection of Ward's family since the rooted multi-vertex fields satisfy Ward's OPEs. 
In Subsection~\ref{ss: 1-leg fields} we introduce the so-called one-leg operator as a special form of rooted vertex field. 
Correlation functions of fields in the image of extended OPE family under the insertion of one-leg operators form a collection of SLE martingale-observables. 
See the next two sections.

\ms\subsection{Non-chiral vertex fields} \label{ss: VV}
We define (non-chiral) \emph{vertex fields} in $\FF_{(b)}$ as OPE-exponentials of the bosonic field $\Phi\equiv\Phi_{(b)}.$ 
More precisely, for $\alpha\in\C$, we have 
$$\VV^\alpha\equiv \VV^\alpha_{(b)}=e^{*\alpha\Phi}=\sum_{n=0}^\infty\frac{\alpha^n}{n!}\Phi^{*n}.$$
To apply Ward's equation to a differential, it is sufficient (and necessary) to verify Ward's OPE.
(See Proposition~5.5 in \cite{KM11}.)
We call a differential \emph{primary} if it satisfies Ward's OPE (or Ward's equation). 
\begin{prop}
A vertex field $\VV^\alpha_{(b)}$ is primary with conformal dimensions
$$\lambda=-\alpha^2/2+i\alpha b,\qquad \lambda_*=-\alpha^2/2-i\alpha b.$$
\end{prop}

\begin{proof}
As in the chordal case, we have
$$\VV^\alpha_{(b)}=e^{\alpha\varphi}\VV^\alpha_{(0)}=e^{\alpha\varphi}C^{\alpha^2}e^{\odot\alpha\Phi_{(0)}},\qquad(\varphi:=\E\,\Phi=-2b\arg w'/w).$$
The non-random field $\varphi$ is an imaginary part of a pre-pre-Schwarzian form and $C$ is the conformal radius. 
In terms of a conformal map $w:(D,q)\to (\D,0),$ we have
$$(C\,\|\,\id_D)(z) = \frac{1-|w(z)|^2}{|w'(z)|},$$ 
where $\id_D$ is the identity chart of $D.$ 
Thus the conformal radius is a $[-1/2,-1/2]$-differential and $\VV^\alpha_{(b)}$ is a $[\lambda,\lambda_*]$-differential.
By definition, the non-chiral vertex fields belong to the OPE family $\FF_{(b)}$ of $\Phi_{(b)}.$
It is a subcollection of Ward's family $\FF(A_{(b)},\overline{A_{(b)}}).$ (See Proposition~4.8 in \cite{KM11}.) 
\end{proof}

\ms While the expression for non-chiral vertex fields in the upper half-plane does not depend on $b$ in the chordal case, the expression for $\VV^\alpha_{(b)}$ in $(\D,0)$-uniformization depends on $b$ in the radial case.
For example, the 1-point function is
$$\E\VV^{\alpha}=(1-|z|^2)^{\alpha^2}\left(\frac{\bar z}{z}\right)^{i\alpha b},$$
and the 2-point function is
\begin{equation} \label{eq: 2ptfcn4V}
\E~\VV^\alpha(z_1)\VV^\alpha(z_2)=
((1-|z_1|^2)(1-|z_2|^2))^{\alpha^2}
\left|\frac{1-z_1\bar z_2}{z_1- z_2}\right|^{2\alpha^2}
\left(\frac{\bar z_1}{z_1}\frac{\bar z_2}{z_2}\right)^{i\alpha b}.
\end{equation}
On the other hand, the conformal properties of the vertex fields, as well as their Virasoro fields $T,$ depend on the central charge.

\ms\subsection{Chiral bosonic fields} \label{ss: Phi+}
We define a \emph{bi-variant} field $\Phiplus$ by 
$$\Phiplus_{(b)}(z,z_0)\Big(=\big\{\Phiplus_{(b)}(\gamma)=\int_\gamma J_{(b)}(\zeta)\,d\zeta\big\}\Big)=\Phiplus_{(0)}(z,z_0)+ib\log \frac{w'}{w}\Big|^{z}_{z_0},$$
where $\gamma$ is a curve from $z_0$ to $z$ and $w$ is any conformal map $(D,q)\to (\D,0).$
Then the values of $\Phiplus_{(b)}$ are multivalued functionals. 
It is natural to add the chiral bosonic field $\Phiplus$ to the OPE family $\FF_{(b)}$ of $\Phi_{(b)}$ since the following Ward's OPE holds in the both variables:
\begin{equation}\label{eq: OPE4Phi+}
T(\zeta)\Phiplus(z, z_0)\sim \frac{ib}{(\zeta-z)^2}+\frac{J(z)}{\zeta-z}, \qquad \zeta\to z,
\end{equation}
where $T\equiv T_{(b)}$ and $\Phiplus\equiv\Phiplus_{(b)}.$
(Similar statement holds for $z_0.$) 
See \cite{KM11} for the meaning of \eqref{eq: OPE4Phi+}.
To show Ward's OPE \eqref{eq: OPE4Phi+} for $\Phiplus$ we just need to integrate 
Ward's OPE for $J,$
$$T(\zeta)J(\eta)\sim \frac{2ib}{(\zeta-\eta)^3}+\frac{J(\zeta)}{(\zeta-\eta)^2},$$
with respect to $\eta.$ 
\begin{eg*}
We have
\begin{equation} \label{eq: G^+}
\E[\Phiplus_{(0)}(z, z_0)\Phi_{(0)}(z_1)]=2(\Gchiral(z,z_1)-\Gchiral(z_0,z_1)),
\end{equation}
where $\Gchiral$ is the \emph{complex} Green's function,
$$2\Gchiral(z,z_1)=G(z,z_1)+i\wt G(z,z_1).$$
Here $\wt G$ is the harmonic conjugate of the Green's function.
In terms of a conformal map $w:(D,q)\to (\D,0),$ we have
$$\Gchiral(z,z_1)=\frac12\log\frac{1-w(z)\overline{w(z_1)}}{w(z)-w(z_1)}.$$
\end{eg*}

\ms \subsection{Chiral bi-vertex fields} \label{ss: V}
As in \cite{KM11} (see Section 7.2) we define chiral bi-vertex fields as the normalized exponentials of the chiral bosonic fields. 
They can be interpreted as the OPE exponentials of $\Phiplus.$
Due to the difficulties with the definition of OPE multiplications of chiral fields, we just state the definition of chiral vertex fields and then verify that these multivalued fields are $\Aut(D,q)$-invariant holomorphic primary fields.

\ms For $z\ne z_0\in D\sm\{q\}$ and $\alpha\in\C$, we define \emph{chiral bi-vertex fields} $V^\alpha(z,z_0)$ by 
\begin{align*}
V^\alpha(z,z_0)&\equiv V_{(b)}^\alpha(z,z_0)=\{V_{(b)}^\alpha(\gamma)\}=\left(\frac{w'(z)w'(z_0)}{(w(z)-w(z_0))^2}\right)^{-\alpha^2/2} e^{\odot \alpha\Phiplus_{(b)}(z,z_0)}\\&=
\left(\frac{w'(z)w'(z_0)}{(w(z)-w(z_0))^2}\right)^{-\alpha^2/2}~\left(\frac{w'(z)}{w'(z_0)}\right)^{i\alpha b}~\left(\frac{w(z_0)}{w(z)}\right)^{i\alpha b}~e^{\odot \alpha\Phiplus_{(0)}(z,z_0)},
\end{align*}
where $w$ is any conformal map $(D,q)\to (\D,0)$ and $\gamma$ is a curve from $z_0$ to $z.$
This expression arises from calculation of the 1-point functions and the interaction term in the multi-vertex fields $\OO^{(-i\alpha,i\alpha)}(z,z_0).$ 
See the next subsection. 

\ms The field $V_{(b)}^\alpha(z,z_0)$ is an $\Aut(D,q)$-invariant holomorphic differential of conformal dimension
$$\lambda=-\frac{\alpha^2}2+i\alpha b \qquad (\lambda_0=-\frac{\alpha^2}2-i\alpha b)$$
with respect to $z$ (with respect to $z_0,$ respectively).
Proposition~\ref{TO} tells us that Ward's OPEs hold for the multi-vertex fields (with the neutrality condition). 
In particular, Ward's OPEs hold for $V_{(b)}^\alpha(z,z_0).$
Thus it is natural to add $V_{{(b)}}^\alpha$ to $\FF_{(b)}.$
See \cite{KM11} for the meaning of global Ward's identities involving chiral vertex fields.
Ward's equations also hold for chiral bi-vertex fields with similar interpretation.

\ms\subsection{Multi-vertex fields} \label{ss: O} 
In this subsection we introduce the \emph{formal (1-point) fields} $\Phiplus_{(0)}$ and $\Phiminus_{(0)}$ (with $b=0$). 
As 1-point fields, they are not Fock space fields. 
However, the multi-vertex fields described in terms of $\Phiplus_{(0)}$ and $\Phiminus_{(0)}$ are Fock space fields under the neutrality condition. 

\ms By definition, $\Phiplus_{(0)}$ and $\Phiminus_{(0)}$ satisfy the following formal rules:
$$\Phiplus_{(0)}(z,z_0) = \Phiplus_{(0)}(z)-\Phiplus_{(0)}(z_0),\qquad \Phi_{(0)}(z) = \Phiplus_{(0)}(z) + \Phiminus_{(0)}(z), \qquad \Phiminus_{(0)} = \overline{\Phiplus_{(0)}}$$
with formal correlations
\begin{equation}\label{eq: formal E}
\E[\Phiplus_{(0)}(z)\Phiplus_{(0)}(z_0)] = \log\frac1{w-w_0},\qquad \E[\Phiplus_{(0)}(z)\Phiminus_{(0)}(z_0)] = \log(1-w\bar w_0),
\end{equation}
where $w = w(z), w_0 = w(z_0).$
We define the formal dual boson $\wt\Phi_{(0)}$ by 
$$\wt\Phi_{(0)} = -i(\Phiplus_{(0)} - \Phiminus_{(0)})$$
so that $2\Phiplus_{(0)} = \Phi_{(0)} + i\wt\Phi_{(0)}, 2\Phiminus_{(0)} = \Phi_{(0)} - i\wt\Phi_{(0)},$ and $\wt\Phi_{(0)}(z,z_0) = \wt\Phi_{(0)}(z)-\wt\Phi_{(0)}(z_0).$
We also write $\Phi_{(0)}(z,z_0) = \Phi_{(0)}(z)-\Phi_{(0)}(z_0).$ 

\ms We define the \emph{multi-vertex field} $\OO^{(\bfs\sigma,\bfs\sigma_*)}$ by 
\begin{equation} \label{eq: OO}
\OO^{(\bfs\sigma,\bfs\sigma_*)}(\bfs z)
=\prod_j M^{(\sigma_j,\sigma_{j*})}(z_j)\, \prod_{j<k} I_{j,k}(z_j,z_k) \,
e^{\odot i\sum \sigma_j\Phiplus_{(0)}(z_j) - \sigma_{j*}\Phiminus_{(0)}(z_j)},
\end{equation}
where $M^{(\sigma_j,\sigma_{j*})} 
=(w_j')^{h_j }(\overline{w_j'})^{h_{j*}}w_j^{\mu_j}\bar w_j^{\mu_{j*}}(1-|w_j|^2)^{\sigma_j\sigma_{j*}}$ 
with the dimensions and exponents
\begin{equation} \label{eq: OOdim}
h_j = \sigma_j^2/2-b\sigma_j, \quad 
h_{j*} = \sigma_{j*}^2/2-b\sigma_{j*}, \quad 
\mu_j = b\sigma_j, \quad 
\mu_{j*} = b\sigma_{j*}.
\end{equation}
The interaction terms $I_{j,k}$ are given by 
\begin{equation} \label{eq: Ijk}
I_{jk}(z_j,z_k) = (w_j-w_k)^{\sigma_j\sigma_k}(\bar w_j-\bar w_k)^{\sigma_{j*}\sigma_{k*}} (1-w_j\bar w_k)^{\sigma_j\sigma_{k*}} (1-\bar w_j w_k)^{\sigma_{j*}\sigma_k}.
\end{equation}
We will explain this formula later. 
The formal field $\OO^{(\bfs\sigma,\bfs\sigma_*)}$ is a (multivalued) Fock space field if and only if the neutrality condition
$$\sum (\sigma_j +\sigma_{j*}) = 0$$
holds.
Indeed, the formal field $\sum \sigma_j\Phiplus_{(0)}(z_j) - \sigma_{j*}\Phiminus_{(0)}(z_j)$ can be decomposed into a Fock space part
$$\sum_j \sigma_j \Phi_{(0)}(z_1) + \sum_{j>1} \sigma_j\Phiplus_{(0)}(z_j,z_1) - \sigma_{j*}\Phiminus_{(0)}(z_j,z_1)$$ 
and a formal part 
$$-\sum (\sigma_j +\sigma_{j*}) \Phiminus_{(0)}(z_1).$$ 
We will not study the sophisticated monodromy structure of Wick's part 
$$e^{\odot i\sum \sigma_j\Phiplus_{(0)}(z_j) - \sigma_{j*}\Phiminus_{(0)}(z_j)}$$ 
of multi-vertex fields in this paper.

\ms\subsubsec{Formal vertex fields} 
To explain the definition of mixed vertex fields, we introduce formal vertex fields. 
We write $\OO^{(\sigma)}(z,z_0)$ for the bi-vertex field $V^{i\sigma}(z,z_0).$ 
We formally define 
$$\OO^{(\sigma)} = (w')^h w^\mu \,e^{\odot i\sigma\Phiplus_{(0)}},$$
where $h = \sigma^2/2-b\sigma$ and $\mu = b\sigma.$
Then the bi-vertex fields can be expressed in terms of the formal vertex fields:
$$
\OO^{(\sigma)}(z,z_0) =\OO^{(\sigma)}(z)\OO^{(-\sigma)}(z_0) 
= M^{(\sigma)}(z)M^{(-\sigma)}(z_0) I^{(\sigma)}(z,z_0) \,e^{\odot i\sigma\Phiplus_{(0)}(z,z_0)},
$$
where $M^{(\sigma)}(z) := \E\,\OO^{(\sigma)}(z)$ and the interaction terms $I^{(\sigma)}(z,z_0)$ are given by
$$I^{(\sigma)}(z,z_0) :=\,e^{\sigma^2\E[\Phiplus_{(0)}(z)\Phiplus_{(0)}(z_0)]}=(w-w_0)^{-\sigma^2}.$$

\ms\subsubsec{1-point vertex fields}
We formally define the basic (1-point) fields $\OO^{(\sigma,\sigma_*)}$ by
$$\OO^{(\sigma,\sigma_*)}=\OO^{(\sigma)}\overline{\OO^{(\bar\sigma_*)}}= M^{(\sigma)}\overline{M^{(\bar\sigma_*)}} I^{(\sigma,\sigma_*)} 
\,e^{\odot (i\sigma\Phiplus_{(0)}-i\sigma_*\Phiminus_{(0)})},$$
where the interaction terms $I^{(\sigma,\sigma_*)}$ are given by
$$I^{(\sigma,\sigma_*)} (z):= e^{\sigma\sigma_*\E[\Phiplus_{(0)}(z)\Phiminus_{(0)}(z)]} = (1-|w|^2)^{\sigma\sigma_*}.$$
Then their 1-point functions $M^{(\sigma,\sigma_*)}:= \E\,\OO^{(\sigma,\sigma_*)} $ are 
$$
M^{(\sigma,\sigma_*)}(z) 
=(w')^h (\overline{w'})^{h_*}w^\mu \bar w^{\mu_*}(1-|w|^2)^{\sigma\sigma_*},
$$
where the exponents are given by 
$$
h = \sigma^2/2-b\sigma, \quad 
h_* = \sigma_*^2/2-b\sigma_*, \quad 
\mu = b\sigma, \quad 
\mu_* = b\sigma_*.
$$

\ms\subsubsec{Multiplication} 
As we see in Subsection~\ref{ss: main} we can view $\bfs\sigma$ (and $\bfs\sigma_*$) as a map $\bfs\sigma: D\sm\{q\} \to \R$ which takes the value $0$ at all points except for finitely many points. 
The tensor product of two multi-vertex fields is a multi-vertex field. 
Indeed, 
$$\OO^{(\bfs\sigma_1,\bfs\sigma_{1*})} \OO^{(\bfs\sigma_2,\bfs\sigma_{2*})} = \OO^{(\bfs\sigma,\bfs\sigma_*)},$$
where $\bfs\sigma = \bfs\sigma_1 + \bfs\sigma_2$ and $\bfs\sigma_*=\bfs\sigma_{1*}+\bfs\sigma_{2*}.$
As a binary operation, multiplication of basic formal fields is commutative and associative, up to monodromy.
Multi-vertex fields are interpreted as product of basic 1-point fields, i.e.,
$$
\OO^{(\bfs\sigma,\bfs\sigma_*)}(\bfs z)\equiv\OO^{(\sigma_1,\sigma_{1*})}(z_1)\, \OO^{(\sigma_2,\sigma_{2*})}(z_2) \cdots \OO^{(\sigma_n,\sigma_{n*})}(z_n)$$
and the interaction terms $I_{j,k}$ (see \eqref{eq: Ijk}) are 
$$I_{j,k}(z_j,z_k) = e^{-\E\,
(\sigma_j\Phiplus_{(0)}(z_j) - \sigma_{j*}\Phiminus_{(0)}(z_j))\,
(\sigma_k\Phiplus_{(0)}(z_k) - \sigma_{k*}\Phiminus_{(0)}(z_k))}.$$
Moreover, the basic 1-point vertex field can be viewed as product of formal vertex fields as follows:
$$\OO^{(\sigma,\sigma_*)} = \OO^{(\sigma,0)} \OO^{(0,\sigma_*)}.$$
Multiplication may have different meanings. 
Suppose $\sigma+\sigma_* = \sigma'+\sigma_*'=0.$ 
The product in 
$$\OO^{(\sigma+\sigma',\sigma_*+\sigma_*')}(z) =\OO^{(\sigma,\sigma_*)}(z)\OO^{(\sigma',\sigma_*')}(z)$$
is an OPE product while the product in 
$$\OO^{(\sigma+\sigma',\sigma_*+\sigma_*')}(z,z') =\OO^{(\sigma,\sigma_*)}(z)\OO^{(\sigma',\sigma_*')}(z'),\qquad (z\ne z')$$
is a tensor product. 

\ms It is natural to expect that Ward's OPEs hold for the multi-vertex fields $\OO^{(\bfs\sigma,\bfs\sigma_*)}(\bfs z)$ since they can be viewed as the OPE exponentials of $i\sum \sigma_j\Phiplus_{(b)}(z_j) - \sigma_{j*}\Phiminus_{(b)}(z_j),$ $(\Phiplus_{(b)} = \Phiplus_{(0)}+ib\log w'/w, \Phiminus_{(b)} = \overline{\Phiplus_{(b)}}).$ 
As we mentioned in Section~\ref{ss: Radial CFT} we now add the multi-vertex fields $\OO^{(\bfs\sigma,\bfs\sigma_*)}$ to the OPE family $\FF_{(b)}$ of $\Phi_{(b)}.$

\ms \begin{prop} \label{TO}
As $\zeta\to z_j,$
$$T(\zeta)\OO(\bfs z)\sim h_{j\phantom{*}}\frac{\OO(\bfs z)}{(\zeta-z_j)^2} + \frac{\pa_{z_j}\OO(\bfs z)}{\zeta-z_j},\qquad
T(\zeta)\bar\OO(\bfs z)\sim \bar h_{j*}\frac{\bar\OO(\bfs z)}{(\zeta-z_j)^2} + \frac{\pa_{z_j}\bar\OO(\bfs z)}{\zeta - z_j},
$$
where $T \equiv T_{(b)}$ and $\OO\equiv\OO^{(\bfs\sigma,\bfs\sigma_*)}.$
\end{prop}

\begin{proof}
We verify Ward's OPEs by Wick's calculus.
We consider the simplest case $n=2$ only. 
There is no difficulty in extending the operator product expansions to the $n$-point vertex fields.
Since $\OO$ is a differential, it is sufficient to verify Ward's OPEs for $\OO$ in just one chart, e.g., in the disc uniformization.
In the global chart of $(\D,0),$ we have
$$\Phi=\Phi_{(0)} + 2b\arg z,\quad \Phiplus=\Phiplus_{(0)} -ib\log z,\quad J=J_{(0)} + j\quad (j=-\frac{ib}z),$$
and
$$T=T_{(0)}-jJ_{(0)} + ib\partial J_{(0)} + (ib\pa j -\frac12j^2),\quad T_{(0)}=-\frac12 J_{(0)}\odot J_{(0)}.$$
Consider the special case $b=0$ first. 
All we need to show is
\begin{align*}
T_{(0)}(\zeta)\,\OO(z,z_0) &\sim \frac{\sigma^2}{2} \frac{\OO(z,z_0)}{(\zeta-z)^2} +i\sigma\frac{J_{(0)}(z)\odot \OO(z,z_0)}{\zeta-z} \\ 
&+\Big(-\sigma\sigma_*\frac{\bar z}{1-|z|^2}+\sigma\sigma_0\frac1{z-z_0}-\sigma\sigma_{0*}\frac{\bar z_0}{1-z\bar z_0}\Big)\frac{\OO(z,z_0)}{\zeta-z}
\end{align*}
as $\zeta\to z.$
Denote
$$\mathcal{J}\equiv \mathcal{J}(\zeta;z,z_0):=\E[J(\zeta)(i\sigma\Phiplus_{(0)}(z)-i\sigma_*\Phiminus_{(0)}(z)+i\sigma_0\Phiplus_{(0)}(z_0)-i\sigma_{0*}\Phiminus_{(0)}(z_0))].$$
Then by differentiating \eqref{eq: formal E} 
$$\mathcal{J} =-i\Big(\frac{\sigma}{\zeta-z}+\frac{\sigma_*}{\zeta-z^*}+\frac{\sigma_0}{\zeta-z_0}+\frac{\sigma_{0*}}{\zeta-z_0^*}\Big),$$
where $z^*=1/\bar z, z_0^* = 1/\bar z_0$ are symmetric points of $z,z_0$ with respect to the unit circle $\pa\D,$ respectively.
By Wick's calculus, 
$$T_{(0)}(\zeta)\,\OO = T_{(0)}(\zeta)\odot \OO - \mathcal{J}\,J_{(0)}(\zeta)\odot \OO -\frac12 \,\mathcal{J}^2\,\OO. 
$$
As $\zeta\to z,$ $-\mathcal{J}^2\,\OO/2$ satisfies 
$$-\frac12 \,\mathcal{J}^2\, \OO\sim \Big(\frac12\frac{\sigma^2}{(\zeta-z)^2}+\frac{\sigma}{\zeta-z}\big(\frac{\sigma_*}{z-z^*}+\frac{\sigma_0}{z-z_0}+\frac{\sigma_{0*}}{z-z_0^*}\big)\Big) \OO(z,z_0).$$
For $b\ne 0,$ we just need to show that
\begin{equation} \label{eq: TVpf1}
\Sing_{\zeta\to z}\frac1\zeta\, J_{(0)}(\zeta)\, \OO(z,z_0)=-\frac{i\sigma}{z}\frac{\OO(z,z_0)}{\zeta-z}
\end{equation}
and
\begin{equation} \label{eq: TVpf2}
\Sing_{\zeta\to z}\pa J_{(0)}(\zeta)\, \OO(z,z_0)= i\sigma\,\frac{\OO(z,z_0)}{(\zeta-z)^2}.
\end{equation}
We derive \eqref{eq: TVpf2} by differentiating the singular operator product expansion
$$\Sing_{\zeta\to z}J_{(0)}(\zeta)\,\OO(z,z_0)=-i\sigma\,\frac{\OO(z,z_0)}{\zeta-z}.$$
This also gives \eqref{eq: TVpf1}.
Similarly, we get Ward's OPE for $\bar\OO.$
\end{proof}

\ms\subsection{Rooted vertex fields} \label{ss: O*} 
We now normalize the multi-vertex field 
\begin{align*}
\OO^{(\sigma,\sigma_{*})}(z)\, \OO^{(\sigma_0,\sigma_{0*})}(z_0) =
&M^{(\sigma,\sigma_{*})}(z)\, M^{(\sigma_0,\sigma_{0*})}(z_0) I(z,z_0) \\
&e^{\odot (i\sigma\Phiplus_{(0)}(z)-i\sigma_*\Phiminus_{(0)}(z)+i\sigma_0\Phiplus_{(0)}(z_0)-i\sigma_{0*}\Phiminus_{(0)}(z_0))},
\end{align*}
so that the limit exists as $z_0\to q.$
We denote this limit $\OO^{(\sigma,\sigma_*;\sigma_q,\sigma_{q*})} (z)$ (with $\sigma_q = \sigma_0$ and $\sigma_{q*} = \sigma_{0*}$).
We now define the rooted (formal) vertex fields (and explain how to arrive to this definition  below), 
\begin{align} \label{eq: O}
\OO^{(\sigma,\sigma_*;\sigma_q,\sigma_{q*})}(z) = &(w')^h (\overline{w'})^{h_*}w^\nu \bar w^{\nu_*}(w'_q)^{h_q}(\overline{w'_q})^{h_{q*}}(1-|w|^2)^{\sigma\sigma_*}\\
&e^{\odot (i\sigma\Phiplus_{(0)}(z)-i\sigma_*\Phiminus_{(0)}(z)+i\sigma_q\Phiplus_{(0)}(q)-i\sigma_{q*}\Phiminus_{(0)}(q))}, \nonumber
\end{align}
where the exponents are given by 
$$
h = \sigma^2/2-b\sigma, \, \, 
\nu = \mu + \sigma\sigma_q = (b+\sigma_q)\sigma, \,\, 
h_q = \mu_0 + h_0 =\sigma_q^2/2
$$
and 
$$
h_* = \sigma_*^2/2-b\sigma_*, \, \, 
\nu_* = \mu_* + \sigma_*\sigma_{q*} =(b+\sigma_{q*})\sigma_*, \, \, 
h_{q*} = \mu_{0*} + h_{0*} = \sigma_{q*}^2/2.
$$
Note that not only the conformal dimensions $h_0,h_{0*}$ but also the exponents $\mu_0, \mu_{0*}$ in 
$$M^{(\sigma_0,\sigma_{0*})}(z_0) = (w_0')^{h_0 }(\overline{w_0'})^{h_{0*}}w_0^{\mu_0}\bar w_0^{\mu_{0*}}(1-|w_0|^2)^{\sigma_0\sigma_{0*}}$$ 
contribute to the conformal dimensions $h_q,h_{q*}$ at $q.$

\ms \textbullet~\textbf{Rooting rules.} This definition can be arrived to by applying the following rooting rules to the mixed vertex field $\OO^{(\sigma,\sigma_{*})}(z)\, \OO^{(\sigma_0,\sigma_{0*})}(z_0)$: 
\begin{enumerate}
\item the term
$$(1-|w_0|^2)^{\sigma_0\sigma_{0*}}$$ 
in $M^{(\sigma_0,\sigma_{0*})}(z_0)$ and all terms in the interaction term 
$$I(z,z_0)=(w-w_0)^{\sigma\sigma_0}(\bar w-\bar w_0)^{\sigma_*\sigma_{0*}} (1-w\bar w_0)^{\sigma\sigma_{0*}} (1-\bar w w_0)^{\sigma_*\sigma_0}$$
involving the point $z_0$ are ignored;
\item the term 
$$w_0^{\mu_0}\bar w_0^{\mu_{0*}}$$ 
in $M^{(\sigma_0,\sigma_{0*})}(z_0)$ is replaced by 
$$(w'_q)^{\mu_q}(\overline{w'_q})^{\mu_{q*}},$$
where $\mu_q = \mu_0$ and $\mu_{q*} = \mu_{0*}.$
\end{enumerate}
Equivalently, 
\begin{equation} \label{eq: normalization}
\OO^{(\sigma,\sigma_*;\sigma_q,\sigma_{q*})}(z) = \lim_{\ve\to0} z_\ve^{-\mu_q} \bar z_\ve^{-\mu_{q*}}\OO^{(\sigma,\sigma_{*})}(z)\, \OO^{(\sigma_q,\sigma_{q*})}(z_\ve) ,
\end{equation}
where the point $z_\ve$ is at distance $\ve$ from $q$ in the chart $\phi.$
If the neutrality condition holds, then the vertex fields $\OO^{(\sigma,\sigma_*;\sigma_q,\sigma_{q*})}$ are well-defined Fock space fields. 
Furthermore, they are primary and invariant with respect to $\Aut(D,q).$
Thus they satisfy the equations of $\FF_{(b)}$-theory. 

\ms There is no difficulty in extending the definition of $1$-point rooted vertex fields to $n$-point fields. 
By definition,
\begin{align} \label{eq: OO*}
\OO^{(\bfs\sigma,\bfs\sigma_*;\tau,\tau_*)}(\bfs z) = &(w'_q)^{h_q}(\overline{w'_q})^{h_{q*}} \prod_{j} \,M_j \prod_{j<k} I_{j,k} \\
&e^{\odot i(\tau\Phiplus_{(0)}(q) - \tau_*\Phiminus_{(0)}(q) + \sum \sigma_j\Phiplus_{(0)}(z_j) - \sigma_{j*}\Phiminus_{(0)}(z_j))},
\nonumber
\end{align}
where the interaction terms $I_{jk}$ are given by \eqref{eq: Ijk} and 
\begin{equation} \label{eq: Mj}
M_j =(w_j')^{h_j} (\overline{w_j'})^{h_{j*}}w_j^{\nu_j} \bar w_j^{\nu_{j*}} (1-|w_j|^2)^{\sigma_j\sigma_{j*}}
\end{equation}
with $\nu_j = \mu_j + \sigma_j\tau = (b+\tau)\sigma_j, \nu_{j*} = \mu_{j*} + \sigma_{j*}\tau_* = (b+\tau_*)\sigma_{j*}.$
Rooted vertex fields $\OO^{(\bfs\sigma,\bfs\sigma_*;\tau,\tau_*)}$ have dimensions $[\bfs h, \bfs h_*;h_q,h_{q*}]:$
$$h_j = \frac{\sigma_j^2}2-\sigma_jb, \qquad h_{j*} = \frac{\sigma_{j*}^2}2-\sigma_{j*}b; \qquad h_q = \frac{\tau^2}2,\qquad h_{q*} = \frac{\tau_*^2}2.$$

\ms \subsection{One-leg fields} \label{ss: 1-leg fields} 
In Subsection~\ref{ss: main}, we introduce Wick's exponentials 
$$e^{\odot -\frac a2 \wt\Phi_{(0)}(p,q)} = e^{\odot ia(\Phiplus_{(0)}(p,q) + \frac12 \Phi_{(0)}(q))}, \qquad(p\in\pa D)$$
as insertion fields. 
In this subsection, we normalized them properly to construct a one-parameter family of $\Aut(D,q)$-invariant primary fields
$$\oneleg:=\OO^{(a,0;-a/2,-a/2)}$$
in $\FF_{(b)}.$
By definition~\eqref{eq: O}, 
\begin{equation} \label{eq: one-leg} 
\oneleg(z) = \Big(\frac{w'(z)}{w(z)}\Big)^h |w'(q)|^{2h_q}e^{\odot ia(\Phiplus_{(0)}(z,q) + \frac12\Phi_{(0)}(q))},
\end{equation}
where $w$ is a conformal map from $(D,q)$ to $(\D,0).$
They are $[h,0]$-differentials with respect to $z$ and $[h_q,h_{q*}]$-differentials with respect to $q.$ 
The conformal dimensions of $\oneleg$ are given by 
$$h = a^2/2-ab, \qquad h_q =h_{q*}= a^2/8.$$
The conformal dimensions $[h_q,h_{q*}]$ at $q$ do not depend on the central charge.

\ms Write $H_q = h_q + h_{q*}$ for the rooted vertex field $\OO.$ 
We now define the effective dimension $H_q^\eff \equiv H_q^\eff(\OO)$ of the rooted vertex field $\OO$ by 
$$H_q^\eff := h_q + h_{q*}-b^2$$
and the effective rooted vertex field $\OO^\eff$ by 
$$\OO^\eff := \OO\,\PP(q)$$
so that $H_q^\eff(\OO) = H_q(\OO^\eff),$
where $\PP(q)$ is the ``puncture operator." 
Recall that it is a $[-b^2/2,-b^2/2]$-differential with respect to $q$ and $\PP(q)\equiv1$ in the identity chart of $\D.$
See Remark in Subsection~\ref{ss: main}.
If the parameters $a$ and $b$ are related to the SLE parameter $\kappa$ as 
$$a = \sqrt{2/\kappa},\qquad b = a(\kappa/4-1),$$
then the effective one-leg operator $\oneleg^\eff$ has conformal dimensions 
$$h(\oneleg^\eff) = \frac{a^2}2-ab = \frac{6-\kappa}{2\kappa},\,\,H_q(\oneleg^\eff) = \frac{a^2}4-b^2 = \frac{h(\oneleg^\eff) }6 + \frac{c}{12} = \frac{(\kappa-2)(6-\kappa)}{8\kappa},$$
where $c$ is the central charge. 
Traditional notations for the conformal dimensions $h(\oneleg^\eff)$ and $H_q(\oneleg^\eff)/2$ are $h_{1,2}$ and $h_{0,1/2},$ respectively.

\ms \subsection{Ward's identities for rooted vertex fields} \label{ss: Ward4O*}
The rooted multi-vertex field $\OO^{(\sigma,\sigma_*;\sigma_q,\sigma_{q*})}$ can be represented in terms of a renormalization procedure, see \eqref{eq: normalization}. 
Thus the properties of $\OO^{(\sigma,\sigma_*;\sigma_q,\sigma_{q*})}$ can be derived from those of bi-vertex fields $\OO^{(\sigma,\sigma_{*})}(z)\, \OO^{(\sigma_0,\sigma_{0*})}(z_0).$
In particular, Ward's OPEs survive under the normalization procedure. 
\begin{prop} \label{TO*}
As $\zeta\to z_j,$
$$T(\zeta)\OO(\bfs z)\sim h_{j\phantom{*}}\frac{\OO(\bfs z)}{(\zeta-z_j)^2} + \frac{\pa_{z_j}\OO(\bfs z)}{\zeta-z_j},$$
where $T \equiv T_{(b)}$ and $\OO\equiv\OO^{(\bfs\sigma,\bfs\sigma_*;\tau,\tau_*)}.$
Similar operator product expansion (with $\bar h_{j*}$) holds for $\bar\OO.$ 
\end{prop}

\ms Also Ward's identities for $\OO^{(\sigma,\sigma_*;\sigma_q,\sigma_{q*})}$ can be obtained from Ward's identities for $\OO^{(\sigma,\sigma_{*})}(z)\, \OO^{(\sigma_0,\sigma_{0*})}(z_0).$

\ms\begin{prop} \label{Ward4V*}
If $v$ is a non-random local holomorphic vector field smooth up to the boundary, and if $v(q)=0,v'(q)=1,$ then Ward's identities
\begin{equation} \label{eq: Ward4V*}
\E\big[W(v)\,\OO(\bfs z)\big] 
=\E\,\LL_v\big[\OO(\bfs z) \big]
+(h_q+h_{q*})\,\E\,\big[\OO(\bfs z) \big]
\end{equation}
hold true for $\OO \equiv \OO^{(\bfs\sigma,\bfs\sigma_*;\tau,\tau_*)}$ with the neutrality condition provided that $z_j$'s are in $D_\hol(v)\sm\{q\}.$ 
\end{prop}
The Lie derivative operators $\LL_v$ do not apply to $q,$ i.e., $\LL_v = \LL(v,D\sm\{q\}).$
\begin{proof} 
It suffices to prove \eqref{eq: Ward4V*} for $\OO = \OO^{(\sigma,\sigma_*;\sigma_q,\sigma_{q*})}.$
We write $\OO(z,z_\ve)$ for $\OO^{(\sigma,\sigma_{*})}(z)\, \OO^{(\sigma_q,\sigma_{q*})}(z_\ve)$ and $\OO(z)$ for $\OO^{(\sigma,\sigma_{*};\sigma_q,\sigma_{q*})}(z),$ respectively.
Let $[\lambda_q,\lambda_{q*}]$ be the conformal dimensions of $\OO^{(\sigma_q,\sigma_{q*})}(z_\ve)$ and let $\mu_q = b\sigma_q, \mu_{q*} = b\sigma_{q*}.$
Since Ward's identities hold for $\OO(z,z_\ve),$ all we need to verify is the following equation: 
$$\lim_{\ve\to0}\frac{\LL_v \OO(z,z_\ve)}{z_\ve^{\mu_q}\bar z_\ve^{\mu_{q*}}}=\LL_v \OO(z) + (h_q+h_{q*})\,\OO(z).$$
Since both $\OO(z,z_\ve)$ and $\OO(z)$ are differentials, it is sufficient to show it in $(\D,0)$-uniformization.
It is obvious that 
$$\lim_{\ve\to0}\frac{\LL_v(z)\OO(z,z_\ve)}{z_\ve^{\mu_q}\bar z_\ve^{\mu_{q*}}}=\LL_v \OO(z).$$
The only non-trivial terms that contribute to the limit of
$$\frac{\LL_v^+(z_\ve)\OO(z,z_\ve)}{z_\ve^{\mu_q}\bar z_\ve^{\mu_{q*}}} = \frac{v(z_\ve)\pa_{z_0}\OO(z,z_\ve)+\lambda_q v'(z_\ve)\OO(z,z_\ve)}{z_\ve^{\mu_q}\bar z_\ve^{\mu_{q*}}} $$ 
are
$$\Big(\mu_q\frac{v(z_\ve)}{z_\ve} + \lambda_q v'(z_\ve)\Big) \frac{\OO(z,z_\ve)}{z_\ve^{\mu_q}\bar z_\ve^{\mu_{q*}}}.$$
Since $v(q)=0,v'(q)=1,$ we have 
$$\lim_{\ve\to0}\frac{\LL_v^+(z_\ve)\OO(z,z_\ve)}{z_\ve^{\mu_q}\bar z_\ve^{\mu_{q*}}} =(\mu_q+\lambda_q) \OO(z) = h_q\, \OO(z).$$
Similar equation (with $h_{q*}$) holds for $\LL_v^-(z_\ve)\OO(z,z_\ve).$
\end{proof}

The following proposition is immediate from Propositions~\ref{represent T} and \ref{Ward4V*}.

\begin{prop} \label{T=L+eff}
For a rooted vertex field $\OO \equiv \OO^{(\bfs\sigma,\bfs\sigma_*;\tau,\tau_*)}$ with the neutrality condition,
\begin{equation} \label{eq: T=L+eff}
2\zeta^2\,\E[T_{(b)}(\zeta)\,\OO] = \E[(\LL_{v_\zeta}^++\LL_{v_{\zeta^*}}^-)\,\OO] + H_q^\eff(\OO)\, \E[\OO],\qquad (\zeta\in\D),
\end{equation}
where all fields are evaluated in the identity chart of $\D\sm\{0\}.$
\end{prop}

\ms Ward's equation~\eqref{eq: T=L+eff} can be rewritten as 
\begin{equation} \label{eq: T=effL}
2\zeta^2\,\E[T_{(b)}(\zeta)\,\OO] = \E[(\LL^+(v_\zeta,\D)+\LL^-(v_{\zeta^*},\D))\,\PP(q)\,\OO],\qquad (\zeta\in\D),
\end{equation}
where $\PP(q)$ is the ``puncture" operator. 
See Subsection~\ref{ss: 1-leg fields} or Remark in Subsection~\ref{ss: main} for its definition. 
To interpret this, we need to compute $\E\,T_v(0)\,\OO.$

\ms \begin{lem} 
For a meromorphic vector field $v$ with a zero at $0$ and a rooted vertex field $\OO \equiv \OO^{(\bfs\sigma,\bfs\sigma_*;\tau,\tau_*)}$ with the neutrality condition,
$$\frac1{2\pi i}\,\E\oint_{(0)} vT_{(b)}\,\OO = \Big(-\frac{b^2}2 + h_q\Big)\,v'(0)\,\E\,\OO.$$
\end{lem}

\begin{proof}
The proof is straightforward by Wick's calculus. 
In the identity chart of $\D\sm\{0\},$ 
$$T_{(b)}(\zeta) = -\frac12 J_{(0)}(\zeta) \odot J_{(0)} (\zeta)+ ib\Big(\pa_\zeta+\frac1\zeta\Big) J_{(0)}(\zeta) -\frac{b^2}2\frac1{\zeta^2}.$$
The term $ib(\pa_\zeta+1/\zeta) J_{(0)}(\zeta)$ has no contribution to $\E\,T_v(0)\,\OO$ since $\E\, J_{(0)}(\zeta)\Phiminus(q)=0$ and the 2-point function $\E\, J_{(0)}(\zeta)\Phiplus(q)=-1/\zeta$ is annihilated by the differential operator $\pa_\zeta + 1/\zeta.$ 
It is easy to see that the contribution of the term $-b^2/(2\zeta^2)$ to $\E\,T_v(0)\,\OO$ is $-b^2/2\,v'(0)\,\E\,\OO.$
Lemma now follows since 
\begin{align*}
\frac1{2\pi i}\,&\E\oint_{(0)} v(\zeta)(J_{(0)}(\zeta) \odot J_{(0)} (\zeta))\,\OO\,d\zeta \\
&= 
\frac1{2\pi i}\oint_{(0)} \E\,\zeta v'(0)\,\big(\E\, J_{(0)}(\zeta)(i\tau\Phiplus(q) - i\tau_*\Phiminus(q)) \big)^2\,\OO\,d\zeta = -\tau^2\,v'(0)\,\E\,\OO
\end{align*}
and $h_q = \tau^2/2.$
\end{proof}

\ms Let us interpret \eqref{eq: T=effL}. 
Recall that 
$\LL_v^+X = v\pa X + \lambda v'X$ for a $[\lambda,\lambda_*]$-differential $X.$ 
Since $\OO$ is a $[h_q,h_{q*}]$-differential with respect to $q,$ 
$$\LL_v^+(q)\,\OO= h_q\,v'(q)\,\OO$$
for a smooth vector field $v$ with a zero at $q.$
Thus Ward's equation~\eqref{eq: T=L+eff} can be rewritten as \eqref{eq: T=effL} and Proposition~\ref{represent T} holds in the sense of correlations with $\OO.$ 

\ms Since the vector fields $$v_z(\zeta) = \zeta\frac{z+\zeta}{z-\zeta} \qquad ({\rm in}\;\D)$$ have a zero at 0 and $v_z'(0)=(v_z^\#)'(0)=1,$ we can apply Proposition~\ref{T=L+eff} to the vector fields $v_z$ and derive Ward's equations under the insertion of rooted vertex field $V\equiv\OO^{(\sigma,\sigma_*;\sigma_q,\sigma_{q*})}$ in terms of the modes $L_{-2},L_{-1},L_0$ of the Virasoro field $T.$
Recall the modes $L_n$ of the Virasoro field $T_{(b)}:$
\begin{equation} \label{eq: Ln}
L_n(z) := \frac1{2\pi i}\oint_{(z)} (\zeta-z)^{n+1}\,T(\zeta)\,d\zeta.
\end{equation}
The following form of Ward's equations is useful to prove Theorem~\ref{main O}. 

\ms \begin{prop} \label{Ward4VX}
For a 1-point rooted vertex field $V,$ and a rooted multi-vertex field $\OO$ with the neutrality condition, 
\begin{align*}
\E\,&V(z)\star\LL_{v_z}^+\OO + \E\,\LL_{v_{z^*}}^-(V(z)\star\OO)\\ 
&=2z^2\,\E[(L_{-2}V)(z)\star\OO] +3z\,\E[(L_{-1}V)(z)\star\OO] \\& 
+(h(V)-H_q^\eff(V\star\OO))\,\E[V(z)\star\OO] , 
\end{align*}
where $z$ is different from any nodes of $\OO$ and all fields are evaluated in the identity chart of $\D\sm\{0\}.$
\end{prop}

\begin{rmks*}
\renewcommand{\theenumi}{\alph{enumi}}
{\setlength{\leftmargini}{1.7em}
\begin{enumerate}
\ss\item The Lie derivative operators $\LL_{v}^\pm$ do not apply to the point $q,$ i.e., $\LL_{v}^\pm = \LL^\pm(v,\D\sm\{0\});$
\ss\item 
Leibniz's rule applies to $\star$-products, i.e., 
$$\LL_{v}^+(V\star\OO) = (\LL_{v}^+V)\star\OO + V\star(\LL_{v}^+\OO);$$
\item The fields $V(z)\star\LL_{v_z}^+\OO$ can be obtained by applying the rooting rules to the tensor product of a bi-vertex field and the Lie derivative of a multi-vertex field. 
Since $V,\OO,$ and $V\star\OO$ are differentials $(V\star\pa_j\OO = \pa_j(V\star\OO)),$ one can express $V(z)\star\LL_{v_z}^+\OO$ explicitly. 
\end{enumerate}}
\end{rmks*}

\begin{proof}
By Proposition~\ref{T=L+eff},
\begin{align*}
\E\,&V(z)\star \LL^+_{v_\zeta}\OO + \E\,\LL^-_{v_{\zeta^*}} [V(z)\star\OO] \\
&=2\zeta^2\,\E\,[T(\zeta) V(z)\star\OO]-\E (\LL^+_{v_\zeta} V)(z)\star\OO- H_q^\eff(V\star\OO)\,\E\,[V(z)\star\OO].
\end{align*}
It follows from \eqref{eq: sing OPE} that 
$$\lim_{\zeta\to z} 2\zeta^2\,T(\zeta) V(z)-\LL^+_{v_\zeta} V(z) = 2z^2\, T*V(z) + 3z \,T*_{-1}V (z)+ T*_{-2}V(z).$$ 
Since $V$ is primary, $T*_{-2}V = hV, \,(h\equiv h(V)).$
\end{proof}

\ms Ward's equations under the insertion of $V(z) \,(z\in\D)$ in Proposition~\ref{Ward4VX} can be extended to the case that $z\in\pa\D.$ 
In the next section we combine this form of Ward's equations with the level two degeneracy equations (Proposition~\ref{TVd2V} below) for one-leg fields $\oneleg$ to prove Theorem~\ref{main O}.

\ms \subsection{Level two degeneracy equations} We now briefly review the level two degeneracy equations for current primary fields and derive the level two degeneracy equations for one-leg fields $\oneleg.$
Let $\{J_n\}$ and $\{L_n\}$ (see \eqref{eq: Ln}) denote the modes of the current field $J$ and the Virasoro field $T$ in $\FF_{(b)}$ theory, respectively: 
$$J_n(z):=\frac1{2\pi i}\oint_{(z)}(\zeta-z)^{n} J(\zeta)~d\zeta.$$
Recall that $X\in\FF_{(b)}$ is (Virasoro) primary if $X$ is a $[\lambda,\lambda_*]$-differential; equivalently
\begin{equation} \label{eq: T primary}
L_{\ge1}X=0,\qquad L_0X=\lambda X, \qquad L_{-1}X=\partial X,
\end{equation}
and similar equations hold for $\bar X.$
Here, $L_{\ge k}X = 0$ means that $L_nX = 0$ for all $n\ge k.$ 
See Section 5.4 in \cite{KM11} for more details.
A (Virasoro) primary field $X$ is called \emph{current primary} if 
\begin{equation} \label{eq: J primary1}
J_{\ge1}X = J_{\ge1}\bar X = 0,
\end{equation}
and
\begin{equation} \label{eq: J primary2}
J_0X = -iqX, \quad J_0\bar X = i\bar q_*\bar X
\end{equation}
for some numbers $q$ and $q_*$ (``charges" of $X$). 
Let us recall the characterization of level two degenerate current primary fields.

\ms \begin{prop}[Proposition~E.2 in \cite{KM11}] \label{degeneracy2} 
Let $V$ be a current primary field in $\FF_{(b)},$ and let $q,q_*$ be charges of $V.$
If $2q(b+q) = 1,\, \eta = -1/(2q^2),$ then
$$(L_{-2}+\eta L_{-1}^2)V =0.$$
\end{prop}
As rooted chiral vertex fields, the one-leg operators $\oneleg\equiv\OO^{(a,0;-a/2,-a/2)}$ are Virasoro primary holomorphic fields of conformal dimension $\lambda= a^2/2-ab$ and current primary with charges $q = a, q_*=0.$
The singular OPE
$$J_{(0)}(\zeta)\oneleg(z) \sim -ia\, \frac1{\zeta-z}\, \oneleg(z) $$
follows from Wick's calculus 
\begin{align*}
J_{(0)}(\zeta)\oneleg(z) &= J_{(0)}(\zeta)\odot \oneleg(z)\\
&+ia\E[J_{(0)}(\zeta)(\Phiplus_{(0)}(z)-\frac12\Phiplus_{(0)}(q)+\frac12\Phiminus_{(0)}(q))]\,\oneleg(z)
\end{align*}
 and 
$$\E[J_{0}(\zeta)(\Phiplus_{(0)}(z)-\frac12\Phiplus_{(0)}(q)+\frac12\Phiminus_{(0)}(q))] = -\frac1{\zeta-z} + \frac1{2\zeta},\qquad (\textrm{in }\D).$$
On the other hand, the operator product expansion of $J_{(0)}(\zeta)$ and $\overline{\oneleg(z)} $ has no singular terms. 
Thus \eqref{eq: J primary1} and \eqref{eq: J primary2} follow.
Proposition~\ref{degeneracy2} implies the following level two degeneracy equation for one-leg operators $\oneleg.$

\ms \begin{prop} \label{TVd2V} 
We have
$$T_{(b)}*\oneleg=\frac1{2a^2}\partial^2 \oneleg,$$
provided that $2a(a+b)=1.$
\end{prop}

\ms As mentioned before, we will use this level two degeneracy equation to establish the connection between the radial SLE theory and conformal field theory.

\ms\section{Connection between radial SLE theory and CFT} \label{sec: SLE}

In Subsection~\ref{ss: one-leg} we discuss how the insertion of one-leg operator $\oneleg$ acts on Fock space fields in $D\sm\{q\}.$
In Subsections~\ref{ss: BPZ4X} and \ref{ss: SLEMO} we derive the boundary version of Belavin-Polyakov-Zamolodchikov type equations (BPZ-Cardy equations) from  Ward's equations and level two degeneracy equations to prove that correlators of fields (without nodes at $q$) in the OPE family $\FF_{(b)}$ of $\Phi_{(b)}$ under the insertion of one-leg operators are martingale-observables in radial SLEs.
Examples of such radial SLE martingale-observables are discussed in the last subsection. 
Vertex observables (with covariance at $q$) will be considered in the next section.

\bs\subsection{One-leg operators} \label{ss: one-leg}

The insertion of one-leg operator 
$$\oneleg(p):=\OO^{(a,0;-a/2,-a/2)}(p) \quad (p\in \partial D)$$ 
produces an operator 
$$\XX\mapsto\wh\XX$$
acting on Fock space functionals/fields by the formula
\begin{gather} \label{eq: BC}
\wh\Phi(z) =\Phi(z)+a\arg \frac{(1-w(z))^2}{w(z)},\\
\wh\Phiplus(z_1,z_0) =\Phiplus(z_1,z_0)-\frac{ia}2\log \frac{(1-w(z))^2}{w(z)}+\frac{ia}2\log \frac{(1-w(z_0))^2}{w(z_0)},\label{eq: BC+}
\end{gather}
(where $w:(D, p,q)\to (\D, 1,0)$ is a conformal map)
and the rules 
\begin{equation} \label{eq: BCrules}
\pa\XX\mapsto\pa\wh\XX,\qquad \bp\XX\mapsto\bp\wh\XX,\qquad \alpha\XX+\beta\YY\mapsto \alpha\wh\XX+\beta\wh\YY,\qquad \XX\odot\YY\mapsto\wh\XX\odot\wh\YY,
\end{equation}
for Fock space functionals $\XX$ and $\YY$ in $D\sm\{q\}.$

\ms Let $\wh\varphi := \E\,\wh\Phi.$
In the $(\D, 1,0)$-uniformization, this 1-point function $\wh\varphi$ is a unique (multivalued) harmonic function (defined on $\D\sm\{0\}$) satisfying the following conditions (see \cite{Dubedat09}):
\renewcommand{\theenumi}{\alph{enumi}}
{\setlength{\leftmargini}{1.7em}
\begin{enumerate}
\smallskip\item $\wh\varphi$ increases by $2b$ times the winding on the boundary, with an additional jump of $-2\pi a$ at $p=1;$
\smallskip\item $\wh\varphi$ has monodromy $-2\pi a+4\pi b$ around $q=0;$
\smallskip\item $\wh\varphi(z)=O(1)$ near $q=0$.
\end{enumerate}}

\ms We denote by $\wh \FF_{(b)}$ the image of $\FF_{(b)}$ under this correspondence.
Then fields in $\wh \FF_{(b)}$ are $\Aut(D,p,q)$-invariant because $\arg w/(1-w)^2$ is $\Aut(D,p,q)$-invariant and fields in $\FF_{(b)}$ are invariant with respect to $\Aut(D,q).$

\bs Denote
$$\wh \E[\XX]:=\frac{\E [\oneleg(p)\XX]}{\E [\oneleg(p)]}=\E [e^{\odot ia(\Phiplus_{(0)}(p,q) + \frac12 \Phi_{(0)}(q))}\XX].$$
As in the chordal case, we have the following:
\ms \begin{prop} \label{hat E=E hat}
Let $\wh\XX\in \wh \FF_{(b)}$ correspond to the string $\XX\in \FF_{(b)}$ under the map given by \eqref{eq: BC} -- \eqref{eq: BCrules}.
Then
\begin{equation} \label{eq: hat E=E hat}
\wh \E[\XX]=\E[\wh\XX].
\end{equation}
\end{prop}

\bs
\textbf{Examples:}
\renewcommand{\theenumi}{\alph{enumi}}
{\setlength{\leftmargini}{1.7em}
\begin{enumerate}
\item The current $\wh J$ is a pre-Schwarzian form of order $ib,$
$$\wh J = J + \frac{ia}2\frac{1+w}{w(1-w)}w'=J_{(0)} + ia\frac{w'}{1-w} + \Big(\frac{ia}2-ib\Big)\frac{w'}{w} + ib\frac{w''}{w'}.$$
In the $(\D,1,0)$-uniformization, $\E\,\wh J(z) = ia\dfrac{1}{1-z} + \Big(\dfrac{ia}2-ib\Big)\dfrac{1}{z};$
\ms \item \label{eg: T hat} The Virasoro field $\wh T$ is a Schwarzian form of order $c/{12},$
\begin{align*}
\wh T &= -\dfrac12\wh J*\wh J + ib\pa\wh J \\
&= A_{(0)} - \wh\jmath J_{(0)} + ib\pa J_{(0)} + \frac{c}{12}S_w + h_{1,2}\frac{w'^2}{w(1-w)^2} + h_{0,1/2}\frac{w'^2}{w^2},
\end{align*}
where $h_{1,2}=a^2/2-ab$ and $h_{0,1/2}={a^2}/8-b^2/2.$
In the $(\D,1,0)$-uniformization, $\E\,\wh T(z) = \dfrac{h_{1,2}}{z(1-z)^2} + \dfrac{h_{0,1/2}}{z^2};$
\ms \item The non-chiral vertex field $\wh\VV^\alpha$ is a $[-\alpha^2/2+i\alpha b,-\alpha^2/2-i\alpha b]$-differential, 
$$\wh\VV^\alpha=e^{2\alpha a\arg (1-w)-\alpha a\arg w}\VV^\alpha = e^{2\alpha a\arg (1-w) -\alpha a\arg w -2\alpha b\arg w'/w}C^{\alpha^2}e^{\odot\alpha\Phi_{(0)}}.$$
In the $(\D,1,0)$-uniformization,
$\wh \E\, \VV^\alpha=(1-|z|^2)^{\alpha^2}e^{2\alpha a\arg (1-z) + \alpha(-a+2b)\arg z};$
\ms\item The bi-vertex field $\wh V^\alpha(z, z_0)$ is a $-\alpha^2/2\pm i\alpha b$-differential with respect to both variables, 
\begin{align*}
\wh V^\alpha(z, z_0)&=
\left(\dfrac{w'(z)w'(z_0)}{(w(z)-w(z_0))^2}\right)^{-\alpha^2/2}\,\left(\dfrac{w'(z)}{w(z)}\dfrac{w(z_0)}{w'(z_0)}\right)^{i\alpha b}\\
&\times\left(\dfrac{w(z)}{(1-w(z))^2}\dfrac{(1-w(z_0))^2}{w(z_0)}\right)^{-i\alpha a/2}\,e^{\odot \alpha\Phiplus_{(0)}(z,z_0)}.
\end{align*}
In the $(\D,1,0)$-uniformization, 
$$\wh \E \,V^\alpha(z, z_0)=(z-z_0)^{\alpha^2}(1-z)^{-i\alpha a} (1-z_0)^{i\alpha a} z^{i\alpha(a/2-b)}z_0^{-i\alpha(a/2-b)}.$$
\end{enumerate}}

\ms The operator $\XX\mapsto\wh\XX$ can be extended to the formal vertex fields, e.g., 
$$\wh\OO^{(\sigma)} = \frac{(1-w)^{a\sigma}}{w^{a\sigma/2}} \OO^{(\sigma)}.$$
Since the interaction terms do not change, we arrive to the definition of $\wh\OO^{(\bfs\sigma,\bfs\sigma_*)} ,$ 
$$\wh\OO^{(\bfs\sigma,\bfs\sigma_*)} = \wh M^{(\bfs\sigma,\bfs\sigma_*)} e^{\odot i\sum \sigma_j\Phiplus_{(0)}(z_j) - \sigma_{j*}\Phiminus_{(0)}(z_j)},$$
where $\wh M^{(\bfs\sigma,\bfs\sigma_*)} = M^{(\bfs\sigma,\bfs\sigma_*)} \prod (1-w_j)^{a\sigma_j} w_j^{-a\sigma_j/2} (1-\bar w_j)^{a\sigma_{j*}} \bar w_j^{-a\sigma_{j*}/2}.$
Thus
\begin{equation} \label{eq: OOhat}
\wh\OO^{(\bfs\sigma,\bfs\sigma_*)}(\bfs z)
=\prod \wh M^{(\sigma_j,\sigma_{j*})}(z_j)\, \prod_{j<k} I_{j,k}(z_j,z_k) \,
e^{\odot i\sum \sigma_j\Phiplus_{(0)}(z_j) - \sigma_{j*}\Phiminus_{(0)}(z_j)},
\end{equation}
where $\wh M^{(\sigma_j,\sigma_{j*})} 
=(w_j')^{h_j }(\overline{w_j'})^{h_{j*}}w_j^{\wh\mu_j}\bar w_j^{\wh\mu_{j*}}(1-w_j)^{a\sigma_j}(1-\bar w_j)^{a\sigma_{j*}} (1-|w_j|^2)^{\sigma_j\sigma_{j*}}$ 
with the exponents
$$
\wh\mu_j = \mu_j - a\sigma_j /2= (b-a/2)\sigma_j, \quad 
\wh\mu_{j*} = \mu_{j*} - a\sigma_{j*}/2 =(b-a/2)\sigma_{j*}.
$$
The interaction terms $I_{j,k}$ are the same as \eqref{eq: Ijk}.

\ms For two multi-vertex fields $\OO_1\equiv\OO^{(\bfs\sigma_1,\bfs\sigma_{1*})}$ and $\OO_2\equiv\OO^{(\bfs\sigma_2,\bfs\sigma_{2*})}$ we have 
$$\wh\OO^{(\bfs\sigma_1,\bfs\sigma_{1*})} \wh\OO^{(\bfs\sigma_2,\bfs\sigma_{2*})} = \wh\OO^{(\bfs\sigma,\bfs\sigma_*)},$$
where $\bfs\sigma = \bfs\sigma_1 + \bfs\sigma_2$ and $\bfs\sigma_*=\bfs\sigma_{1*}+\bfs\sigma_{2*}.$
Thus
$$\wh\OO_1\wh\OO_2 = \wh{(\OO_1\OO_2)}.$$

\ms\subsection{BPZ-Cardy equations on the unit circle} \label{ss: BPZ4X}
For $\xi =e^{i\theta} \in \partial\mathbb{D}$ and the tensor product 
$$X=X_1(z_1)\cdots X_n(z_n)$$
of fields $X_j$ in the OPE family $\FF_{(b)}$ of $\Phi_{(b)}$ ($z_j\in\D$), we denote
$$\wh\E_\xi X=\E [e^{\odot ia(\Phiplus_{(0)}(\xi,q) + \frac12 \Phi_{(0)}(q))}X],\qquad (q=0).$$
Then $\wh\E X=\wh{\E}_\xi X\big|_{\xi=1}.$

\begin{prop} \label{BPZ4X} If $2a(a+b)=1,$ then in the identity chart of $\D,$ we have
\begin{equation} \label{eq: BPZ4X}
-\frac{1}{a^2}\partial_\theta^2 \wh\E_\xi X = \wh\E_\xi[\LL_{v_\xi}X],\qquad v_\xi(z):=z\frac{\xi+z}{\xi-z},
\end{equation}
where $\xi =e^{i\theta} $ and $\pa_\theta$ is the operator of differentiation with respect to the real variable $\theta.$
\end{prop}

\begin{proof} Denote
$$R_\xi\equiv R_\xi(z_1,\cdots, z_n)=\wh{\E}_\xi X.$$
In the disc normalization, the evaluation of one-leg operator $\oneleg=\OO^{(a,0;-a/2,-a/2)}$ at $\xi$ is given by 
$$\oneleg(\xi) = \xi^{-h} e^{\odot ia(\Phiplus_{(0)}(\xi,q) + \frac12 \Phi_{(0)}(q))}, \qquad(h=a^2/2-ab),$$ 
(see \eqref{eq: one-leg}) and 
$R_\xi=\E[\xi^{h}\oneleg(\xi)X].$
Denote
$$R_\zeta\equiv R(\zeta;z_1,\cdots, z_n)=\E [\zeta^h \oneleg(\zeta)X],\qquad (\zeta\in\D).$$

Using Ward's equations (Proposition~\ref{Ward4VX}) and the fact that $\LL^-_{v}\oneleg(\zeta)=0,$ we have 
\begin{align*}
\E\,&\oneleg(\zeta)(\LL_{v_\zeta}^+X+\LL_{v_{\zeta^*}}^-X)\\
&=2\zeta^2\,\E[(L_{-2}\oneleg)(\zeta)X] +3\zeta\,\E[(L_{-1}\oneleg)(\zeta)X] +(h-H_q^\eff)\,\E\,\oneleg(\zeta)X.
\end{align*}
Conformal dimensions $h = a^2/2-ab, H_q^\eff = a^2/4-b^2$ of $\oneleg$ satisfy the so-called fusion rule:
\begin{equation} \label{eq: fusion rule}
h-H_q^\eff = \frac{h^2}{a^2}.
\end{equation}
By the level two degeneracy equations for $\oneleg$ (Proposition~\ref{TVd2V}, $2a^2L_{-2}\oneleg = \pa^2 \oneleg$) and $L_{-1}\oneleg = \pa \oneleg,$ we get
\begin{equation} \label{eq: Cardy0}
\E\,\oneleg(\zeta)(\LL^+_{v_\zeta}X +\LL^-_{v_{\zeta^*}}X)= \frac{1}{a^2}\zeta^{2}\pa_\zeta^2(\zeta^{-h}R_\zeta)+3\zeta\pa_\zeta(\zeta^{-h} R_\zeta)+\frac{h^2}{a^2}\,\zeta^{-h}R_\zeta.
\end{equation}
Equivalently, 
\begin{equation*}
\E\,\zeta^{h}\oneleg(\zeta)(\LL^+_{v_\zeta}X +\LL^-_{v_{\zeta^*}}X) = \frac1{a^2}\zeta^2\pa_\zeta^2R_\zeta + \Big(-\frac{2h}{a^2}+3\Big)\zeta\pa_\zeta R_\zeta + \Big(\frac{2h}{a^2}+\frac{1}{a^2}-3\Big)h\,R_\zeta.
\end{equation*}
It follows from the numerologies ($2a(a+b)=1$ and $h = a^2/2-ab$) that $h = 3a^2/2-1/2$ and 
$$\E\,\zeta^{h}\oneleg(\zeta)(\LL^+_{v_\zeta}X +\LL^-_{v_{\zeta^*}}X)= \frac1{a^2}\big(\zeta^2\pa_\zeta^2R_\zeta + \zeta\pa_\zeta R_\zeta\big).$$
As $\zeta\to\xi = e^{i\theta},$ the right-hand side converges to $-\pa_\theta^2 R_{e^{i\theta}}/a^2.$ 
On the other hand, since $\xi = \xi^*,$ the left-hand side converges to 
$$\E\,[\xi^h \oneleg(\xi)\LL_{v_\xi}X] = \wh\E_\xi[\LL_{v_\xi}X].$$ 
\end{proof}

\ms \textbf{Examples.} (a) Suppose that a 1-point field $X$ in the OPE family $\FF_{(b)}$ of $\Phi_{(b)}$ is a boundary differential of conformal dimension $h.$ 
Then BPZ-Cardy equation \eqref{eq: BPZ4X} reads
$$ \Big(\frac{\kappa}2\pa_\theta^2 + \cot\frac\theta2\pa_\theta -\frac{h}{2}\sec^2\frac\theta2\Big)\wh{\E}_{e^{i\theta}}[X(1)]=0.$$

\ms (b) The 1-point function $R(z) = \E\,\wh T(z)$ of the Virasoro field $\wh T$ is a Schwarzian form of order $c/12.$ 
Recall that
\begin{equation} \label{eq: Lie4S}
\LL_v X = (v\pa + 2v')X + \mu v'''
\end{equation}
for a Schwarzian form $X$ of order $\mu.$ 
Thus
$$\LL_{v_{_1}}R= \Big(z\frac{1+z}{1-z}\pa+2\frac{1+2z-z^2}{(1-z)^2}\Big)R +\frac{c}{(1-z)^4}.$$
Let
$R_\xi(z) = \E_\xi[\wh{T}(z)].$
Then, by rotation invariance, $R_\xi(z) = \xi^{-2} R(z/\xi)$ and
$$\Big(\big(\xi\pa_\xi\big)^2 R_\xi\Big)\Big|_{\xi=1}(z) = z^2\pa^2R(z)+ 5z\pa R(z) + 4 R(z) .$$
It follows from BPZ-Cardy equation \eqref{eq: BPZ4X} that
$$\frac1{a^2}\big(z^2\pa^2 + 5z\pa + 4 \big) R(z) = z\frac{1+z}{1-z} \pa R(z) + 2\frac{1+2z-z^2}{(1-z)^2} R(z) + \frac{c}{(1-z)^4}.$$

\ms (c) We generalize the previous example to the $n$-point function of $\wh T,$
$$R(z_1\cdots,z_n) := \E[\,\wh T(z_1) \cdots \wh T(z_n) \,\|\,\id_\D].$$
Denote $\bfs{z} = (z_1,\cdots,z_n),$ $\bfs{z}_j = (z_1,\cdots,\wh z_j,\cdots,z_n).$
By \eqref{eq: Lie4S} and Leibniz's rule, 
$$\LL_{v_{_1}}R(\bfs{z}) =\sum_{j=1}^n\Big(z_j\frac{1+z_j}{1-z_j}\pa_j+2\frac{1+2z_j-z_j^2}{(1-z_j)^2}\Big)R(\bfs{z}) + c \sum_{j=1}^n\frac{R(\bfs{z}_j)}{(1-z_j)^4}.$$
Setting
$R_\xi(\bfs{z}) = \E_\xi[\,\wh{T}(z_1)\cdots \wh{T}(z_n)\,],$
it follows from conformal invariance with respect to rotation that 
$R_\xi(\bfs{z}) = \xi^{-2n} R(\bfs{z}/\xi).$
This homogeneity property implies 
$$\Big(\big(\xi\pa_\xi\big)^2 R_\xi\Big)\Big|_{\xi=1}(\bfs{z}) = \sum z_jz_k\pa_j\pa_kR(\bfs{z}) + (4n+1) \sum z_j\pa_j R(\bfs{z}) + 4n^2 R(\bfs{z}).$$
By BPZ-Cardy equation \eqref{eq: BPZ4X}, we obtain the following recursive formula:
\begin{align*}
\frac1{2a^2}\big( \sum z_jz_k\pa_j\pa_k + (4n+1) \sum z_j\pa_j + 4n^2 \big) &R(\bfs{z}) \\
= \sum \frac{z_j}2 \frac{1+z_j}{1-z_j} \pa_j R(\bfs{z}) + \frac{1+2z_j-z_j^2}{(1-z_j)^2} &R(\bfs{z}) + \frac{c}2 \sum \frac{R(\bfs{z}_j)}{(1-z_j)^4}.
\end{align*}
See \cite{FW02} for the corresponding equation in the chordal case with $c=0.$

\ms\subsection{SLE martingale-observables} \label{ss: SLEMO}
We restate Theorem~\ref{main X} as follows.

\begin{thm} \label{MO}
If $X_j$'s are in the OPE family $\FF_{(b)}$ of $\Phi_{(b)},$ then the non-random fields
$$M(z_1,\cdots, z_n)=\wh\E [X_1(z_1)\cdots X_n(z_n)]$$
are martingale-observables for $\SLE_\kappa.$
\end{thm}

\begin{proof}
Denote 
$$R_\xi(z_1,\cdots, z_n)\equiv\wh\E_\xi[X_1(z_1)\cdots X_n(z_n)].$$
By conformal invariance, the process 
$$M_t(z_1,\cdots, z_n)=M_{D_t,\gamma_t,q}(z_1,\cdots, z_n)$$ 
is represented by 
$$M_t = m(\xi_t,t), \qquad m(\xi,t) =\big(R_\xi\,\|\,g_t^{-1}\big), \qquad \xi_t = e^{i\sqrt\kappa B_t},$$
where $g_t:(D_t,\gamma_t,q)\to(\D,\xi_t,0)$ is the radial SLE map. 
Since the function $m(\xi,t)$ is smooth in both variables, we can apply It\^o's formula to the process $m(\xi_t,t).$ 
The drift term of $dM_t$ is equal to
$$-\frac\kappa2\Big((\xi\pa_\xi)^2\Big|_{\xi=\xi_t}~m(\xi,t)\Big)\,dt+L_t\, dt,$$
where
$$
L_t:=\frac d{ds}\Big|_{s=0}\big(R_{\xi_t}\,\|\,g_{t+s}^{-1}\big)
=\frac d{ds}\Big|_{s=0}\big(R_{\xi_t}\,\|\,g_t^{-1}\circ f_{s,t}^{-1}\big),\qquad (f_{s,t} = g_{t+s}\circ g_t^{-1}).
$$
The time-dependent flow $f_{s,t}$ satisfies the following differential equation: 
$$\frac{d}{ds}f_{s,t}(\zeta) = f_{s,t}(\zeta)\frac{\xi_{t+s}+f_{s,t}(\zeta)}{\xi_{t+s}-f_{s,t}(\zeta)}.$$
Thus we have 
$$f_{s,t} = \id + sv_{\xi_t} + o(s)\qquad(\textrm{as } s\to0),\qquad v_{\xi}(z)=z\frac{\xi+z}{\xi-z}.$$
Since the fields in the OPE family $\FF_{(b)}$ of $\Phi_{(b)}$ depend smoothly on local charts, we can represent $L_t$ in terms of Lie derivative:
$$L_t=\big(\LL_{v_{\xi_t}}R_{\xi_t}\,\|\,g_{t}^{-1}\big).$$
By BPZ-Cardy equation \eqref{eq: BPZ4X}, we get
$$L_t=\frac1{a^2}\big((\xi\pa_\xi)^2\Big|_{\xi=\xi_t} R_{\xi}\,\|\,g_{t}^{-1}\big).$$
Thus the drift term of $dM_t$ vanishes.
\end{proof}
In the next section we will consider the fields with a node at $0,$ e.g., the rooted multi-vertex fields.

\ms\subsection{Examples of SLE martingale-observables} \label{ss: EgsMO}
In this subsection we present examples of radial SLE martingale-observables including Schramm-Scheffield's obsevables, Friedrich-Werner's formula, and the restriction formula in the radial case.

\ms \textbf{Example (Schramm-Scheffield's obsevables)} In the chordal case, the 1-point functions of the bosonic fields
$$\wh\varphi(z) = \wh\E[\Phi_{D,p,q}(z)] = 2a\arg w(z) -2b\arg w'(z), \quad w:(D,p,q)\to(\H,0,\infty)$$
were introduced as SLE martingale-observables by Schramm and Sheffield, see \cite{SS10}.
Similarly, the 1-point functions $\wh\varphi= \wh\E[\Phi_{D,p,q}]$ of the bosonic fields in the radial conformal field theory we develop in this paper are martingale-observables of radial SLEs:
$$\wh\varphi(z) = \wh\E[\Phi_{D,p,q}(z)] = 2a\arg (1-w(z)) -a\arg w(z) -2b\arg \frac{w'(z)}{w(z)},$$
where $w:(D,p,q)\to(\D,1,0)$ is a conformal map.
By It\^o's calculus, we have 
\begin{align*}
\wh\varphi_t(z) &= \wh\E[\Phi_{D_t,\gamma_t,q}(z)]= 2a\arg (1-w_t(z)) -a\arg w_t(z) -2b\arg \frac{w_t'(z)}{w_t(z)} \\
&= 2a\arg (1-w(z)) -a\arg w(z) -2b\arg \frac{w'(z)}{w(z)} +\sqrt2\int_0^t\Re\, \frac{1+w_s(z)}{1-w_s(z)}~dB_s.
\end{align*}
One can use the $2$-point martingale-observables
$$\wh\E[\Phi(z_1)\Phi(z_2)]=2G(z_1,z_2)+\wh\varphi(z_1)\wh\varphi(z_2)$$
or Hadamard's variation formula
\begin{equation} \label{eq: Hadamard}
d G_{D_t}(z_1,z_2)=-\,\Re\,\frac{1+w_t(z_1)}{1-w_t(z_1)}\,\Re\,\frac{1+w_t(z_2)}{1-w_t(z_2)}\,dt=-\frac12\,d\langle \wh\varphi(z_1),\wh\varphi(z_2)\rangle_t
\end{equation}
to construct a coupling of radial SLE and the current field such that
$$\E[\,\wh J_{D,p,q}\,|\,\gamma[0,t]\,]=\wh J_{D_t,\gamma_t,q}.$$

\ms \textbf{Example} The 1-point functions of the chiral bi-bosonic fields 
\begin{align*}
M(z,z_0) = \wh\E[\Phiplus_{D,p,q}(z,z_0)] = &-ia\log(1-w(z))\phantom{_0} +\frac{ia}2\log w(z)\phantom{_0} +ib\log \frac{w'(z)}{w(z)}\phantom{_0} \\
&+ia\log(1-w(z_0)) -\frac{ia}2\log w(z_0) -ib\log \frac{w'(z_0)}{w(z_0)}
\end{align*}
are radial SLE martingale-observables with 
$$dM_t(z,z_0) =\frac1{\sqrt2}\Big(\frac{1+w_t(z)}{1-w_t(z)}-\frac{1+w_t(z_0)}{1-w_t(z_0)}\Big)\,dB_t$$
and 
$$d \langle M(z,z_0) \rangle_t = d\log\frac{w_t'(z)w_t'(z_0)}{(w_t(z)-w_t(z_0))^2}.$$
Their exponential martingales
$$e^{\alpha M_t(z,z_0) - \frac12\alpha^2 \langle M(z,z_0) \rangle_t}$$
are bi-vertex observables.

\ms While the processes
$$-ia\log(1-w_t(z))+\frac{ia}2\log w_t(z) +ib\log \frac{w_t'(z)}{w_t(z)}$$
are not local martingales, the 1-point functions of the rooted chiral bosonic fields
$$N(z) = -ia\log(1-w(z))+\frac{ia}2\log \frac{w(z)}{w'(q)} +ib\log \frac{w'(z)}{w(z)}$$
are martingale-observables with
$$dN_t(z) =\sqrt2\,\,\frac{w_t(z)}{1-w_t(z)}\,dB_t$$
and 
$$d \langle N(z) \rangle_t = d\log\frac{w_t'(z)w_t'(0)}{w_t(z)^2}.$$
Their exponential martingales
$$e^{\alpha N_t(z) - \frac12\alpha^2 \langle N(z) \rangle_t}$$
are 1-point observables of the rooted vertex fields.
In Section~\ref{sec: Vertex MO} we will discuss rooted multi-vertex observables.

\ms \textbf{Example (The restriction formula) \cite{Lawler05}} 
Recall the \emph{restriction property} of radial $\SLE_{8/3}:$
{\setlength{\leftmargini}{1.7em}
\begin{itemize}
\ms \item the law of $\SLE_{8/3}$ in $\D$ conditioned to avoid a fixed hull $K$ is identical to the law of $\SLE_{8/3}$ in $\D\sm K;$
\ms \item equivalently, there exist $\lambda$ and $\mu$ such that for all $K,$
$$\P(\SLE_{8/3} \textrm{ avoids }K) = |\psi_K'(1)|^\lambda (\psi_K'(0))^\mu,$$
where $\psi_K$ is the conformal map $(\D\sm K,0)\to (\D,0)$ satisfying $\psi_K'(0)>0.$
(The restriction exponents $\lambda$ and $\mu$ of radial $\SLE_{8/3}$ are equal to $5/8$ and $5/48,$ respectively.)
\end{itemize}}

\ms Let $\kappa\le4.$ 
On the event $\gamma[0,\infty)\cap K = \emptyset,$ a conformal map $h_t:\Omega_t = g_t(D_t\sm K)\to\D$ is defined by
$$h_t = \wt g_t \circ \psi_K \circ g_t^{-1},$$
where $\wt g_t$ is a Loewner map from $\wt D_t = D\sm\wt\gamma[0,t]$ onto $\D,$ and 
$\wt\gamma(t) = \psi_K \circ \gamma(t).$
Let
$$M_t := (\E\, \oneleg^\eff(\gamma_t;D_t\sm K)\,\|\,w_t),$$
where $\oneleg^\eff$ is the effective one-leg operator.
See Remark in Subsection~\ref{ss: main}.
Then 
$$M_t=|h_t'(\xi_t)|^\lambda h_t'(0)^\mu=\Big(\xi_t\frac{h_t'(\xi_t)}{h_t(\xi_t)}\Big)^\lambda h_t'(0)^\mu,\qquad(\xi_t = e^{i\sqrt\kappa B_t}),$$
where exponents are given by 
$$\lambda = h(\oneleg^\eff) \equiv \frac{a^2}2-ab = \frac{6-\kappa}{2\kappa}, \qquad \mu = H_q(\oneleg^\eff)\equiv\frac{a^2}4-b^2 = \frac{(\kappa-2)(6-\kappa)}{8\kappa}.$$ 
Restriction property of radial $\SLE_{8/3}$ follows from the local martingale property of $M_t$ (by optional stopping theorem). 
This is a special case of the following formula:
\begin{equation} \label{eq: restriction}
\textrm{the drift term of } dM_t= -\frac c6 \xi_t^2S_{h_t}(\xi_t) M_t\,dt.
\end{equation}

\ms We present the CFT argument to explain the reason for the appearance of the central charge $c$ and the Schwarzian derivative in \eqref{eq: restriction}.
To prove \eqref{eq: restriction}, denote
$$F(z,t):=(\E\,\oneleg^\eff_{\Omega_t}(z)\,\|\,\id)\,\big(=(\E\,\oneleg^\eff_{\D}\,\|\,h_t^{-1})(z)\big).$$ 
Then
$$M_t= \xi_t^\lambda F(\xi_t,t) = \xi_t^\lambda(\E\,\oneleg^\eff_{\Omega_t}(\xi_t)\,\|\,\id).$$
Since the function $F$ is smooth in both variables, we can apply It\^o's formula to the process $F(\xi_t,t)$ to get the drift term of ${dM_t}/{M_t},$ 
$$\Big(-\frac\kappa2 \lambda^2+\frac{\dot F(\xi_t,t)}{F(\xi_t,t)}-\big(\frac\kappa2+\kappa\lambda\big)\xi_t\frac{F'(\xi_t,t)}{F(\xi_t,t)}-\frac\kappa2\xi_t^2\frac{F''(\xi_t,t)}{F(\xi_t,t)}\Big)\,dt.$$ 
Using the similar method in \cite{KM11}, we represent $\dot F$ in terms of the Lie derivatives:
\begin{align} \label{eq: dotF1}
\dot F(z,t) &= \frac{d}{ds}\Big|_{s=0} (\E\, \oneleg^\eff_{\D}\,\|\, h_{t+s}^{-1}) (z)= \frac{d}{ds}\Big|_{s=0} (\E\, \oneleg^\eff_{\D}\,\|\, h_{t}^{-1}\circ f_{s,t}^{-1}) (z) \\
&= (\E\,\LL(v,\D) \,\oneleg^\eff_{\D}\,\|\,h_t^{-1})(z), \nonumber
\end{align}
where $f_{s,t} = h_{t+s}\circ h_{t}^{-1}$ and
$$(v\,\|\,\id_\D) = \frac{d}{ds}\Big|_{s=0} f_{s,t} = \dot h_t \circ h_t^{-1}.$$

\ss We only need to compute the vector field $v.$
We represent $v$ as the difference of two Loewner vector fields associated to the flows in the domains $\D$ and $\Omega_t.$
Applying the chain rule to $h_t = \wt g_t \circ \psi_K \circ g_t^{-1}$ and computing the capacity changes, we have 
$$\dot h_t(z) = |h_t'(\xi_t)|^2v_{\wt \xi_t}(h_t(z)) - h_t'(z)v_{\xi_t}(z), \quad \Big(v_{\xi}(z)=z\frac{\xi+z}{\xi-z}\Big),$$
where $\wt \xi_t = h_t(\xi_t).$
By the above equation and $(v\,\|\,\id_\D) = \dot h_t \circ h_t^{-1},$ 
\begin{equation} \label{eq: v}
(v\,\|\,\id_\D)(\zeta)=|h_t'(\xi_t)|^2v_{\wt\xi_t}(\zeta) - h_t'(h_t^{-1}(\zeta))v_{\xi_t}(h_t^{-1}(\zeta)).
\end{equation}
It follows from \eqref{eq: dotF1} and \eqref{eq: v} that 
\begin{align*}
\dot F(z,t) &= |h_t'(\xi_t)|^2 h_t'(z)^\lambda \,\E\, \big(\LL(v_{\wt\xi_t},\D) \,\oneleg^\eff_{\D} (h_t(z)\big) \,\|\, \id_\D) \\
&- \Big(\big(\E\,\LL(v_{\xi_t},\Omega_t\sm\{0\}) \,\oneleg^\eff_{\Omega_t}\,\|\,\id_{\Omega_t}\big)(z) + \mu \big(\E\,\oneleg^\eff_{\Omega_t}\,\|\,\id_{\Omega_t}\big)(z) \Big).
\end{align*}
\ss Now we represent $\dot F(\xi_t,t)$ in terms of $F(\xi_t,t), F'(\xi_t,t), F''(\xi_t,t)$ and the Schwarzian derivative $S_{h_t}(\xi_t).$
By Ward's equation, 
\begin{align*} 
\dot F(z,t) &= 2|h_t'(\xi_t)|^2 h_t'(z)^\lambda\wt\xi_t^2(\E\, T_\D(\wt \xi_t)\,\oneleg^\eff_\D (h_t(z))\,\|\,\id_\D)\\
& - (\E\,\LL(v_{\xi_t},\Omega_t\sm\{0\})\,\oneleg^\eff_{\Omega_t}\,\|\,\id_{\Omega_t}\big)(z)-\mu(\E\, \oneleg^\eff_{\Omega_t}\,\|\,\id_{\Omega_t}\big)(z).
\end{align*}
It follows from conformal invariance that
\begin{align*}
\dot F(z,t) &=2\xi_t^2(\E\, T_{\Omega_t}(\xi_t)\,\oneleg^\eff_{\Omega_t}(z)\,\|\,\id_{\Omega_t}) -\frac c{6} \xi_t^2 S_{h_t}(\xi_t)\, (\E\, \oneleg^\eff_{\Omega_t}\,\|\,\id_{\Omega_t}\big)(z)\\
&-(\E\,\LL(v_{\xi_t},\Omega_t\sm\{0\})\,\oneleg^\eff_{\Omega_t}\,\|\,\id_{\Omega_t}\big)(z)-\mu(\E\, \oneleg^\eff_{\Omega_t}\,\|\,\id_{\Omega_t}\big)(z).
\end{align*}
Let us now apply \eqref{eq: sing OPE} (and $T*_{-1}\oneleg^\eff = \pa \oneleg^\eff, T*_{-2}\oneleg^\eff = \lambda \oneleg^\eff$) to the right-hand side of the above equation:
$$\frac{\dot F(\xi_t,t)}{F(\xi_t,t)}
=2\xi_t^2\, \frac{(\E\,T_{\Omega_t}*\oneleg^\eff_{\Omega_t}(\xi_t)\,\|\,\id)}{(\E\,\oneleg^\eff_{\Omega_t}(\xi_t)\,\|\,\id)}
+3\xi_t\,\frac{F'(\xi_t,t)}{F(\xi_t,t)}+\lambda -\mu-\frac c6 \xi_t^2 S_{h_t}(\xi_t) .
$$
By the level two degeneracy for $\oneleg^\eff$ (Proposition~\ref{TVd2V}), we have
\begin{align*}
\frac{(\E\,T_{\Omega_t}* \oneleg^\eff_{\Omega_t}(\xi_t)\,\|\,\id)}{(\E\,\oneleg^\eff_{\Omega_t}(\xi_t)\,\|\,\id)} 
&=\frac{1}{2a^2} \frac{F''(\xi_t,t)}{F(\xi_t,t)}.
\end{align*}
Thus
\begin{align*}
\textrm{the drift term of }\frac{dM_t}{M_t}
&= -\frac c6 \xi_t^2 S_{h_t}(\xi_t) \,dt +\Big(\lambda - \frac\kappa2 \lambda^2 - \mu\Big)\,dt\\
&+ (3-\frac\kappa2 -\kappa\lambda)\frac{\xi_tF'(\xi_t,t)}{F(\xi_t,t)}\,dt + \Big(\frac1{a^2}-\frac\kappa2\Big)\frac{\xi_t^2F''(\xi_t,t)}{F(\xi_t,t)}\,dt.
\end{align*}
The numerologies 
$$a = \sqrt{2/\kappa}, \quad 2a(a+b)=1, \quad \lambda = a^2/2-ab, \quad \mu = a^2/4-b^2$$ 
give \eqref{eq: restriction}.
The exponents $\lambda=h(\oneleg^\eff)$ and $\mu = H_q(\oneleg^\eff)$ satisfy the so-called fusion rule~\eqref{eq: fusion rule}, $\lambda = \mu + \kappa\lambda^2/2.$

\ms We now prove Friedrich-Werner's formula in the radial case.
\begin{thm}\label{radial FW}
Let $z_j = e^{i\theta_j} (\theta_j\in \R)$ all distinct. Then for $a = \sqrt3/2$ and $b = -\sqrt3/6,$
$$\wh\E\,[\,T(z_1)\cdots T(z_n)\,\|\,\id_\D\,] = \lim_{t\to0}\frac{(-1)^n}{(2t)^n} e^{-2i\sum_{j=1}^n\theta_j}
\P(\SLE_{8/3}\textrm{ hits all }[r_te^{i\theta_j},e^{i\theta_j}]),$$
where $r_t = 1-2\sqrt t.$
\end{thm}
\begin{proof}
Let us apply Ward's equations to the function
$$\E\,[\,\wh T(z) \wh T(z_1)\cdots \wh T(z_n)\,\|\,\id_\D\,] = \E\,[\,T(z)\oneleg^\eff(1) T(z_1)\cdots T(z_n)\,\|\,\id_\D\,],$$
by replacing $T(z)$ in the right-hand side with the corresponding Ward's functional.
Denote $\bfs{z} = (z_1,\cdots,z_n),$ $\bfs{z}_j = (z_1,\cdots,\wh z_j\cdots,z_n),$ and
$$R(\xi;\bfs{z}) = \E\,[\,\xi^\lambda \oneleg^\eff(\xi)\,T(z_1)\cdots T(z_n)\,],\qquad(\lambda = a^2/2-ab=5/8).$$
The non-random field $R(\xi;\bfs{z})$ is a boundary differential of conformal dimension $\lambda$ with respect to $\xi$ and of conformal dimension $2$ with respect to $z_j.$
It is also a differential of conformal dimension $\mu=a^2/4-b^2 = 5/48$ with respect to $q=0.$ 
It follows from Ward's equation (see \eqref{eq: T=effL}) for $\oneleg^\eff$ that
\begin{equation} \label{eq: recursion4T}
R(\xi;z,\bfs{z})=\frac{1}{2z^2}\,\LL({v_z},\D)\,R(\xi;\bfs{z}),\quad (\textrm{in }\id_\D),
\end{equation}
where $v_z(\zeta) = \zeta(z+\zeta)/(z-\zeta).$ 

\ms Let 
$$U(\theta_1,\cdots,\theta_n) = \lim_{t\to0}\frac{(-1)^n}{(2t)^n} e^{-2i\sum_{j=1}^n\theta_j}
\P(\SLE_{8/3}\textrm{ hits all }[r_te^{i\theta_j},e^{i\theta_j}])$$
(if the limit exists).
Define the non-random field $T(\xi;z_1,\cdots,z_n)$ as follows:
{\setlength{\leftmargini}{1.7em}
\begin{itemize}
\ms \item $T$ is a boundary differential of conformal dimension $\lambda=5/8$ with respect to $\xi,$ and of conformal dimension $2$ with respect to $z_j;$ 
\ms \item $T$ is a differential of conformal dimension $\mu=a^2/4-b^2 = 5/48$ with respect to $q=0;$ 
\ms \item $(T(e^{i\varphi};e^{i\theta_1},\cdots,e^{i\theta_n})\,\|\,\id_\D) = U(\theta_1-\varphi,\cdots,\theta_n-\varphi).$
\end{itemize}}

\ms We now claim that if the limit $U(\theta_1,\cdots,\theta_n)$ exits then the limit $U(\theta,\theta_1,\cdots,\theta_n)$ exists and 
\begin{equation} \label{eq: FW0}
T(1;z,\bfs{z}) = \frac1{2z^2} \,\LL({v_z},\D)\, T(1;\bfs{z})\qquad z,z_j\in\pa\D.
\end{equation}
By \eqref{eq: recursion4T} and \eqref{eq: FW0}, $T(1;\cdot)$ and $R(1;\cdot)$ satisfy the same recursive equation (see \eqref{eq: recursion4T2} below) and are therefore equal since $T(1;\cdot) = R(1;\cdot) = 1$ for $n=0.$
Thus $$U(\theta_1,\cdots,\theta_n) = R(1;e^{i\theta_1},\cdots,e^{i\theta_n}).$$

\ss To verify the induction argument for existence of the limit $U(\theta_1,\cdots,\theta_n)$ and show \eqref{eq: FW0}, denote $\bfs{\theta} = (\theta_1,\cdots,\theta_n).$
We write $\P(\bfs{\theta})$ for the probability that radial $\SLE_{8/3}$ path hits all segments $[r_te^{i\theta_j},e^{i\theta_j}]\,(1\le j\le n)$ and $\P(\bfs{\theta}\,|\,\neg \theta)$ for the same probability conditioned on the event that the path avoids $[r_te^{i\theta},e^{i\theta}].$
By the induction hypothesis,
\begin{equation} \label{eq: FW1}
\P(\bfs \theta)\approx (-2t)^{n} z_1^2\cdots z_n^2 \, T(1;\bfs z),
\end{equation}
where $z_j = e^{i\theta_j}.$
On the other hand, by the restriction property of radial $\SLE_{8/3},$ we have
\begin{equation} \label{eq: FW2}
1-\P(\theta) =|\psi_t'(1)|^\lambda \, \psi_t'(0)^{\lambda/6}
\end{equation}
and
\begin{equation} \label{eq: FW3}
\P(\bfs{\theta}\,|\,\neg \theta)\approx (-2t)^{n} \, T(\psi_t(1);\psi_t(z_1),\cdots, \psi_t(z_n))\prod_{j=1}^nz_j^2\psi_t'(z_j)^2 , \end{equation}
where $\psi_t$ is a slit map from $(\D\sm [r_t e^{i\theta},e^{i\theta}],0)$ onto $(\D,0)$ with $\psi_t'(0)>0$.
It follows from 
$\P(\theta,\bfs \theta) = \P(\bfs \theta)- \P(\bfs \theta\,|\,\neg \theta)(1-\P(\theta))$\,$(z = e^{i\theta})$ and \eqref{eq: FW1}~--~\eqref{eq: FW3} that 
$$\frac{\P(\theta,\bfs\theta)}{z^2z_1^2\cdots z_n^2(-2t)^{n+1}}$$ is equal to $$\frac{\psi_t'(0)^{\lambda/6} |\psi_t'(1)|^\lambda \prod_{j=1}^n\psi_t'(z_j)^2T(\psi_t(1);\psi_t(z_1),\cdots, \psi_t(z_n))-T(1;\bfs z)}{2tz^2} $$ 
up to $o(t)$ terms.
Thus the limit $U(\theta,\theta_1,\cdots,\theta_n)$ exists. 
Since $\psi_t'(0) = e^t,$
$$T(1;z,\bfs z) = \frac1{2z^2}\Big(\LL(v,\D\sm\{0\}) +\frac{\lambda}6\Big)T(1;\bfs z) = \frac1{2z^2} \,\LL({v},\D)\, T(1;\bfs z),$$
where $v$ is the vector field of flow $\psi_t.$ 
Thus $v = v_z$ and we get \eqref{eq: FW0}.
\end{proof}

\begin{rmk*} 
The formula \eqref{eq: recursion4T} holds for all $\kappa.$
Setting $R(z,\bfs z)\equiv R(1;z,\bfs z),$ 
the formula \eqref{eq: recursion4T} at $\xi = 1$ gives the following recursive formula for $R:$
\begin{align}\label{eq: recursion4T2}
R(z,\bfs{z}) &= \frac1{2z^2}\Big(2n\frac{1+z}{1-z} + 2\lambda \frac{z}{(1-z)^2} + (\frac {a^2}4-b^2) + \frac{1+z}{1-z}\sum_{j=1}^n z_j\pa_j\Big) R(\bfs{z}) \\
&+\frac1{2z^2}\sum_{j=1}^n\Big(z_j\frac{z+z_j}{z-z_j} \pa_j + 2\frac{z^2+2zz_j-z_j^2}{(z-z_j)^2}\Big) R(\bfs{z}) \nonumber \\
&+ \frac c2 \sum_{j=1}^n \frac1{(z-z_j)^4} R(\bfs{z}_j). \nonumber
\end{align}
\end{rmk*}

\ms\section{Multi-Vertex observables} \label{sec: Vertex MO}

In the previous section we prove that correlators of (primary) fields (without node at $q$) in $\wh{\FF}_{(b)}$ are martingale-observables of radial SLEs.
We extend this result to the multi-point vertex fields (with the neutrality condition).
In general, the multi-point (rooted) vertex fields have covariance at $q.$ 
We also discuss examples of vertex observables including Lawler-Schramm-Werner's derivative exponents of radial SLEs on the boundary.

\ms\subsection{1-point vertex fields} \label{ss: o hat}
Applying the rooting rule in Subsection~\ref{ss: O*} to the multi-vertex field $\wh\OO^{(\sigma,\sigma_{*})}(z)\, \wh\OO^{(\sigma_q,\sigma_{q*})}(z_0),$
we arrive to the following definition: 
\begin{align*}
\wh\OO^{(\sigma,\sigma_*;\sigma_q,\sigma_{q*})}(z) = &(w')^h (\overline{w'})^{h _*}w^{\wh\nu} \bar w^{\wh \nu_*}(w'_q)^{\wh h_q}(\overline{w'_q})^{\wh h_{q*}}(1-w)^{a\sigma}(1-\bar w)^{a\sigma_*}(1-|w|^2)^{\sigma\sigma_*}\\
&e^{\odot (i\sigma\Phiplus_{(0)}(z)-i\sigma_*\Phiminus_{(0)}(z)+i\sigma_q\Phiplus_{(0)}(q)-i\sigma_{q*}\Phiminus_{(0)}(q))},
\end{align*}
where the exponents are given by 
$$
h = \sigma^2/2-b\sigma, \quad
\wh\nu = (b-a/2+\sigma_q)\sigma, \quad
\wh h_q = \sigma_q(\sigma_q-a)/2
$$
and 
$$
h _* = \sigma_*^2/2-b\sigma_*, \quad 
\wh\nu_* = (b-a/2+\sigma_{q*})\sigma_*, \quad
\wh h_{q*} = \sigma_{q*}(\sigma_{q*}-a)/2.
$$
(An alternative way to define $\wh\OO^{(\sigma,\sigma_*;\sigma_q,\sigma_{q*})}$ will be discussed in the last subsection.)
If the neutrality condition $\sigma+ \sigma_* + \sigma_q + \sigma_{q*} = 0$ holds, then the formal fields $\wh\OO^{(\sigma,\sigma_*;\sigma_q,\sigma_{q*})}$ are Fock space fields. 
Furthermore, they are $\Aut(D,p,q)$-invariant primary fields. 

\ms In Subsection~\ref{ss: O hat} we prove that the (rooted) multi-vertex fields in the extended  OPE family $\wh\FF_{(b)}$ of $\wh\Phi_{(b)}$ have ``field Markov property." 
In particular, the 1-point non-random fields $\E\,\wh\OO^{(\sigma,\sigma_*;\sigma_q,\sigma_{q*})}$ (which have covariance at $q$ in general) with the neutrality condition are martingales-observables.
It turns out that the neutrality condition is the sufficient and necessary condition for the local martingale property of the formal non-random fields $\E\,\wh\OO^{(\sigma,\sigma_*;\sigma_q,\sigma_{q*})}.$ 

\ms \begin{prop}
The 1-point function $M^{(\sigma,\sigma_*;\sigma_q,\sigma_{q*})}$ of the formal vertex field $\wh\OO^{(\sigma,\sigma_*;\sigma_q,\sigma_{q*})}$ is a martingale-observable if and only if the neutrality condition 
$$\sigma+ \sigma_* + \sigma_q + \sigma_{q*} = 0.$$
holds.
\end{prop}

\begin{proof}
Let 
$$v^{(\sigma)} (\equiv M^{(\sigma,0;-\sigma,0)})=(w')^h w^{\mu-\sigma^2} (w_q')^{\sigma^2/2+a\sigma/2}(1-w)^{a\sigma}$$ 
and 
$$u^{(\alpha,\beta)} = w^\alpha (w_q')^\beta,\qquad (\alpha = \sigma\sigma_q+\sigma^2, \quad \beta = \wh h_q - \sigma(\sigma+a)/2).$$ 
Then $v^{(\sigma)}$ is a vertex observable with 
$$\frac{dv_t^{(\sigma)} \overline{v_t^{(\bar\sigma_*)}}} {v_t^{(\sigma)}\overline{ v_t^{(\bar\sigma_*)}}}=i\sigma\sqrt2\frac{w_t}{1-w_t}\,dB_t-i\sigma_*\sqrt2\frac{\bar w_t}{1-\bar w_t}\,dB_t + 2\sigma\sigma_*\frac{|w_t|^2}{|1-w_t|^2}\,dt.$$
and 
$$M_t = 
{v_t^{(\sigma)}\overline{ v_t^{(\bar\sigma_*)}}} \,\,
{u_t^{(\alpha,\beta)}\overline{u_t^{(\bar\alpha_*,\bar\beta_*)}}} \,\,
{(1-|w_t|^2)^{\sigma\sigma_*}}.$$ 
 By It\^o's calculus, 
\begin{align*}
\textrm{the drift term of }\frac{dM_t}{M_t} &=\big(2\alpha + \sqrt{2\kappa}\sigma(\alpha+\beta-\alpha_*-\beta_*)\big)\frac{w_t}{1-w_t}\,dt\\
&+\big(2\alpha_* - \sqrt{2\kappa}\sigma_*(\alpha+\beta-\alpha_*-\beta_*)\big)\frac{\bar w_t}{1-\bar w_t}\,dt\\
&+\Big(\alpha+\beta+\alpha_*+\beta_*-\frac\kappa2(\alpha+\beta-\alpha_*-\beta_*)^2\Big)\,dt.
\end{align*}
Thus $M$ is a martingale-observable if and only if the neutrality condition holds.
Indeed, the numerologies $(\alpha = \sigma\sigma_q+\sigma^2, \,\beta = \wh h_q - \sigma(\sigma+a)/2)$ give 
$$\alpha+\beta= -\frac a2(\sigma+\sigma_q) + \frac12(\sigma+\sigma_q)^2.$$
Suppose that the neutrality condition holds. 
Then
$$\alpha+\beta-\alpha_*-\beta_*= -\frac a2(\sigma+\sigma_q-\sigma_*-\sigma_{q*})=-a(\sigma+\sigma_q)=-\frac\alpha\sigma\,\sqrt{2/\kappa} = \frac{\alpha_*}{\sigma_*}\,\sqrt{2/\kappa}$$
and
$$\frac\kappa2(\alpha+\beta-\alpha_*-\beta_*)^2=(\sigma+\sigma_q)^2=\alpha+\beta+\alpha_*+\beta_*. $$
Conversely, if the drift term of $dM_t$ vanishes, then 
$$\frac\alpha\sigma= -\frac{\alpha_*}{\sigma_*},$$
which is the neutrality condition.
\end{proof}

\ms \subsubsec{Constant fields}
The simplest examples of 1-point vertex fields are constant fields, i.e., vertex fields with $\sigma=\sigma_* = 0.$ 
By the neutrality condition, $\sigma_{q*} = -\sigma_q.$ 
Since $w_t'(0) = e^{t-i\sqrt\kappa B_t},$
$$M_t^{(0,0;\tau,-\tau)} = e^{\tau^2t+ia\sqrt\kappa\tau B_t} = e^{\tau^2t+i\sqrt2\tau B_t}.$$ 
This is a martingale.

\ms \subsubsec{Real fields} The 1-point vertex fields are real if and only if $\sigma = \sigma_*$ and $\sigma_q = \sigma_{q*}.$
By the neutrality condition, $\sigma_q = -\sigma.$
Thus the only real fields are
$$M^{(\sigma,\sigma;-\sigma,-\sigma)} = e^{t(\sigma^2+a\sigma)}|w'|^{\sigma^2-2b\sigma}|w|^{\sigma(2b-a-2\sigma)}|1-w|^{2a\sigma}(1-|w|^2)^{\sigma^2}.$$
When $\sigma = -a,$ there is no covariance at $q.$ 
In this special case,
$$M^{(-a,-a;a,a)} = \left|\frac{w'}{w}\right|^{1-2/\kappa}\left(\frac{1-|w|^2}{|1-w|^2}\right)^{2/\kappa}.$$

\ms If $\kappa =2,$ then $M^{(-a,-a;a,a)}$ coincides with the Lawler-Schramm-Werner observable 
$$M =\frac{1-|w|^2}{|1-w|^2} = \frac{P_\D(1,w)}{P_\D(1,0)} = \frac{P_D(p,z)}{P_D(p,q)},$$
where $P_D$ is the Poisson kernel of a domain $D.$
As we mentioned in Subsection~\ref{ss: intro MO} this 1-point field is an important observable in the theory of LERW.

\ms If $\kappa=4,$ then 
$$M = \left|\frac{w'}{w}\right|^{1/2}\left(\frac{1-|w|^2}{|1-w|^2}\right)^{1/2}$$
is Beffara's type observable for radial $\SLE_4.$
See \cite{AKL12}.
In the chordal case, Beffara's observables are real martingale-observables of conformal dimensions
$$h=h_*=\frac12-\frac\kappa{16}, \qquad \wh h_q=0$$
with the estimate 
$$\P(z,\ve)\asymp \ve^{1-\frac{\kappa}{8}}~M(z),$$
where
$\P(z,\ve)$ is the probability that the $\SLE_\kappa$ curve ($\kappa<8$) hits the disc at $z$ of size $\ve(\ll1)$ measured in a local chart $\phi.$ 
See \cite{Beffara08}.
In \cite{AKL12}, Feynman-Kac formula is used to construct radial SLE martingale-observables with the desired dimensions.

\ms \subsubsec{1-point vertex fields without covariance at $q$} 
A 1-point vertex field $M$ has no covariance at $q$ if and only if 
$$(\sigma_q,\sigma_{q,*}) = (0,0),\quad (a,0), \quad (0,a), \quad (a,a).$$
The first case is just the non-chiral vertex field
$$M = M^{(\sigma,-\sigma;0,0)}.$$

\subsection{Holomorphic 1-point fields} \label{ss: hol 1pt field}
When $\sigma_*=0,$ the 1-point vertex observable 
$$M^{(\sigma,0;\sigma_q,\sigma_{q*})} = (w'_q)^{\wh h_q}(\overline{w'_q})^{\wh h_{q*}} (w')^h w^{\wh\nu} (1-w)^{a\sigma}$$
is holomorphic.
Recall the expression for the exponents 
$$h = \sigma^2/2-b\sigma, \qquad
\wh\nu = \sigma(b-a/2+\sigma_q), \qquad
\wh h_q =\sigma_q(\sigma_q-a)/2,$$ 
and $\wh h_{q*} = \sigma_{q*}(\sigma_{q*}-a)/2.$

\ms \subsubsec{Holomorphic 1-point fields without spin at $q$} 
A holomorphic 1-point field $M$ has no spin at $q$ if and only if 
$$\wh h_q = \wh h_{q*}.$$
Equivalently,
$$\sigma_q =\sigma_{q*} \quad\textrm{or}\quad \sigma_q +\sigma_{q*}=a.$$

\ms \textbf{Case 1)} $\sigma_q =\sigma_{q*}.$
By the neutrality condition, we have 
$$\bfs\sigma = (\sigma,0;-\sigma/2,-\sigma/2).$$
The holomorphic 1-point field $\wh\OO^{(\sigma,0;-\sigma/2,-\sigma/2)}$ is a generalization of the one-leg operator $\OO^{(a,0;-a/2,-a/2)}.$
In this case, the conformal dimensions of $\wh\OO^{(\sigma,0;-\sigma/2,-\sigma/2)}$ are 
$$h = \frac{\sigma^2}2-b\sigma, \qquad \wh h_q=\frac{\sigma^2}{8} + \frac{a\sigma}4.$$
Thus
$$M_t= e^{2\wh h_qt}\Big(\frac{w_t'}{w_t}\Big)^{\sigma^2/2-b\sigma} \Big(\frac{w_t}{(1-w_t)^2}\Big)^{-{a\sigma}/2}.$$

\ms\begin{eg*}[Derivative exponents on the boundary \cite{LSW01c}]
On the unit circle, (up to constant)
$$M_t(e^{i\theta}) = e^{2\wh h_qt} |w_t'(e^{i\theta})|^h\Big(\sin^2\frac{\theta_t}{2}\Big)^{a\sigma/2},$$
where $w_t(e^{i\theta}) = e^{i\theta_t}.$
Given $h,$ the equation $h = \sigma^2/2-b\sigma $ solves
$$\sigma_\pm = \frac a 4\big(\kappa-4 \pm \sqrt{(\kappa-4)^2+16\kappa h } \big).$$
With the choice of $\sigma = \sigma_+,$ Lawler, Schramm, and Werner proved that 
$M_t(e^{i\theta})$ is a martingale. 
They applied the optional stopping theorem to $M_t(e^{i\theta})$ and used the estimate 
$$\E[M_t(e^{i\theta})] \asymp e^{2\wh h_qt}\, \E[|w'_t(e^{i\theta})|^h \mathbf{1}_{\{\tau_{e^{i\theta}} > t\}}]$$
to derive the derivative exponents:
$$\E[|w'_t(e^{i\theta})|^h \mathbf{1}_{\{\tau_{e^{i\theta}} > t\}}] \asymp e^{-2\wh h_q t} \Big(\sin^2 \frac{\theta}2\Big)^{a\sigma/2}.$$
(Recall that $\tau_z$ is the SLE stopping time, the first time when a point $z$ is swallowed by the hull of SLE, see Subsection~\ref{ss: intro MO}.) 
From the derivative exponent for $\kappa=6,$ they obtained the annulus crossing exponent for $\SLE_6$ and combined it with other exponents to prove Mandelbrot's conjecture that the Hausdorff dimension of the planar Brownian frontier is $4/3$. 
See \cite{LSW01a} and references therein.
\end{eg*}

\ms\begin{eg*}
The field $M$ is a scalar if $\sigma = 2b.$
In this case,
$$M_t= e^{t(\kappa-4)/{8}}\Big(w_t+\frac1{w_t}-2\Big)^{(\kappa-4)/(2\kappa)}.$$
Its derivative $M_1=\pa M$ has the conformal dimensions $[1,0;\wh h_q,\wh h_q].$ 
It is not a vertex observable. 
If $M_1 = M^{(\sigma_1,\sigma_{1*};\sigma_{1q},\sigma_{1q*})}$ with $h _1 = 1,$ then
$\sigma_1 = -2a$ or $\sigma_1 = 2a+2b.$ 
As we will see below, the holomorphic 1-differentials without spin at $q$ are not forms of $\pa M^{(2b,0,-b,-b)}.$
Unlike the chordal case (see Proposition~9.2 in \cite{KM11}), the holomorphic differential observables are not necessarily vertex observables. 
\end{eg*}

\ms\begin{eg*}
If we take $\sigma = 2b-a,$ then $\bfs\sigma = (2b-a,0;-b+a/2,-b+a/2)$ and 
$$h = \frac{6-\kappa}{2\kappa} = -\frac{a\sigma}2,\quad \wh h_q = \frac{(6-\kappa)(2-\kappa)}{16\kappa}, \quad \wh\nu =0.$$ 
In this case, we have 
$$M_t = e^{2\wh h_qt} \Big(\frac{w_t'}{(1-w_t)^2}\Big)^{(6-\kappa)/(2\kappa)}.$$
If $\kappa=2,$ then
$$M = \frac{w'}{(1-w)^2}$$
and its anti-derivative is $-1/(1-w).$ 
See the next example.
\end{eg*}

\ss \textbf{Case 2)} $\sigma_q +\sigma_{q*}=a.$
It follows from the neutrality condition that 
$$\bfs\sigma = (-a,0;\sigma_q,a-\sigma_q).$$
Thus 
$$M = (w')^{(\kappa-2)/(2\kappa)}w^{a(-\sigma_q+a/2-b)}(1-w)^{-2/\kappa}|w_q'|^{\sigma_q^2-a\sigma_q}.$$

\ss\begin{eg*} The fields $M$ have no covariance at $q$ if and only if $\sigma_q = 0$ or $a.$ 
In these cases, we have
$$M^{(-a,0;0,a)} = (w')^{(\kappa-2)/(2\kappa)}w^{(6-\kappa)/(2\kappa)}(1-w)^{-2/\kappa}$$
or
$$M^{(-a,0;a,0)} = \Big(\frac{w'}{w}\Big)^{(\kappa-2)/(2\kappa)}(1-w)^{-2/\kappa}.$$
For example, if $\kappa =6,$ then 
$$M^{(-a,0;0,a)} = \Big(\dfrac{w'}{1-w}\Big)^{1/3},$$
and if $\kappa = 2,$ then both $M^{(-a,0;a,0)}$ and $M^{(-a,0;0,a)}$ produce the same observable 
$$M^{(-a,0;a,0)}=M^{(-a,0;0,a)}=\frac1{1-w}.$$
\end{eg*}

\ms \subsubsec{Holomorphic differentials without spin at $q$} 
A holomorphic 1-point field $M$ is a 1-differential with respect to $z$ if and only if 
$\sigma = -2a$ or $\sigma = 2(a+b).$ 
Furthermore, if $M$ has no spin at $q,$ then $\sigma_q = \sigma_{q*}.$
(If $\sigma = -2a$ or $\sigma = 2(a+b),$ then the other possibility $\sigma_q+\sigma_{q*}=a$ never happens because of the neutrality condition and the numerology $2a(a+b)=1.$)

\ms \textbf{Case 1)} $\sigma = -2a.$ 
By the neutrality condition, we have 
$\bfs\sigma= (-2a,0;a,a)$
and
$$M = w'w^{2/\kappa-1}(1-w)^{-4/k}.$$

\ms\begin{eg*}
If $\kappa=2,$ then
$M^{(-2a,0;a,a)} = w'(1-w)^{-2}$ and its anti-derivative is $-1/(1-w).$ 
See the previous example.
\end{eg*}

\ms\begin{eg*}
If $\kappa=4,$ then
$M^{(-2a,0;a,a)} = w'(1-w)^{-1}w^{-1/2}$ and its anti-derivative is $\log(1+\sqrt w)-\log(1-\sqrt w).$ Its imaginary part is a bosonic observable for a twisted conformal field theory. See \cite{KMZ12}.
\end{eg*}

\ms \textbf{Case 2)} $\sigma = 2(a+b).$ In this case, we have 
$\bfs\sigma= (2a+2b,0;-a-b,-a-b)$ by the neutrality condition, 
and
$$M_t = e^{t(\kappa+4)/8}w_t' w_t^{-3/2}(1-w_t).$$

\ms\subsection{Multi-point observables} \label{ss: O hat}
Rooting procedure applied to 1-point vertex fields in Subsection~\ref{ss: o hat} can be extended to the multi-vertex fields. 
Applying the rooting rules to the multi-vertex field \eqref{eq: OOhat}, we arrive to the definition of the rooted vertex field $\wh\OO^{(\bfs\sigma,\bfs\sigma_*;\tau,\tau_*)}:$ 
\begin{align} \label{eq: OOhat*}
\wh\OO^{(\bfs\sigma,\bfs\sigma_*;\tau,\tau_*)}(\bfs z) &= (w'_q)^{\wh h_q}(\overline{w'_q})^{\wh h_{q*}} \prod_{j} \wh M_j \prod_{j<k} I_{j,k}\\
&\,
e^{\odot i(\tau\Phiplus_{(0)}(q) - \tau_*\Phiminus_{(0)}(q) + \sum \sigma_j\Phiplus_{(0)}(z_j) - \sigma_{j*}\Phiminus_{(0)}(z_j))},
\nonumber
\end{align}
where the interaction terms $I_{jk}$ are given by \eqref{eq: Ijk} and 
$$\wh M_j = (w_j')^{h_j} (\overline{w_j'})^{h_{j*}}w_j^{\wh\nu_j} \bar w_j^{\wh\nu_{j*}}(1-w_j)^{a\sigma_j}(1-\bar w_j)^{a\sigma_{j*}} (1-|w_j|^2)^{\sigma_j\sigma_{j*}}$$
with the exponents
$$\wh\nu_j = \wh\mu_j + \sigma_j\tau = (b-a/2+\tau)\sigma_j,\qquad \wh\nu_{j*} = \wh\mu_{j*} + \sigma_{j*}\tau_* = (b-a/2+\tau_*)\sigma_{j*}.$$ 
The rooted vertex field $\OO^{(\bfs\sigma,\bfs\sigma_*;\tau,\tau_*)}$ has conformal dimensions $[\wh h_q, \wh h_{q*}]$ at $q:$
$$\wh h_q = \frac{\tau^2}2 -\frac{\tau a}2,\qquad \wh h_{q*} = \frac{\tau_*^2}2-\frac{\tau_* a}2.$$
If the neutrality condition holds, then $\wh\OO^{(\bfs\sigma,\bfs\sigma_*;\tau,\tau_*)}$ is a well-defined Fock space field.
Furthermore, it is an $\Aut(D,p,q)$-invariant primary field.

\ms Alternatively, the definition \eqref{eq: OOhat*} can be arrived to by the action of operator $\XX \mapsto \wh\XX$ (produced by the insertion of $\oneleg$) on $\OO^{(\bfs\sigma,\bfs\sigma_*;\tau,\tau_*)}.$
The operator $\XX \mapsto \wh\XX$ acts on formal fields by the formula
$$\Phiplus_{(0)}(z) \mapsto \Phiplus_{(0)}(z) -\frac{ia}2\log\frac{(1-w(z))^2}{w(z)},\qquad 
\Phiplus_{(0)}(q) \mapsto \Phiplus_{(0)}(q) +\frac{ia}2\log w'(q),$$
$$\Phiminus_{(0)}(z) \mapsto \Phiminus_{(0)}(z) +\frac{ia}2\log\frac{(1-\overline{w(z)})^2}{\overline{w(z)}},\qquad 
\Phiminus_{(0)}(q) \mapsto \Phiminus_{(0)}(q) -\frac{ia}2\log \overline{w'(q)},$$
and the rules
$$\pa\XX\mapsto\pa\wh\XX,\qquad \bp\XX\mapsto\bp\wh\XX,\qquad \alpha\XX+\beta\YY\mapsto \alpha\wh\XX+\beta\wh\YY,\qquad \XX\odot\YY\mapsto\wh\XX\odot\wh\YY$$
for formal fields $\XX$ and $\YY$ in $D.$ 
Since the interaction terms do not change by the action of operator $\XX \mapsto \wh\XX,$
$\OO\equiv \OO^{(\bfs\sigma,\bfs\sigma_*;\tau,\tau_*)}$ maps to:
$$\wh\OO = \,(w_q')^{-\tau a/2} (\bar w_q')^{-\tau_* a/2} \prod_j (1-w_j)^{a\sigma_j}w_j^{-a\sigma_j/2} (1-\bar w_j)^{a\sigma_{j*}}\bar w_j^{-a\sigma_{j*}/2}\OO.$$
Thus $\wh\OO =\wh\OO^{(\bfs\sigma,\bfs\sigma_*;\tau,\tau_*)}.$

\ms For rooted multi-vertex fields $\OO \equiv \OO^{(\bfs\sigma,\bfs\sigma_*;\tau,\tau_*)}$ in the extended OPE family $\FF_{(b)}$ of $\Phi_{(b)}$ (with the neutrality condition), we define 
$$\wh\E\, \OO : = \frac{\E\, \oneleg(p)\star\OO}{\E\, \oneleg(p)}, \qquad \wh\E_\xi\, \OO : = \frac{\E\, \oneleg(\xi)\star\OO}{\E\, \oneleg(\xi)}$$
for $\xi\in D$ so that $\wh\E\, \OO = \wh\E_p\, \OO.$
Then as in Proposition~\ref{hat E=E hat}, it can be shown that 
$$\wh\E\, \OO = \E\,\wh\OO.$$
Of course, one can show it directly by the algebra of vertex operators,
$$\OO^{(\bfs\sigma_1,\bfs\sigma_{1*};\tau_1,\tau_{1*})} \star \OO^{(\bfs\sigma_2,\bfs\sigma_{2*};\tau_2,\tau_{2*})} =\OO^{(\bfs\sigma_1+\bfs\sigma_2,\bfs\sigma_{1*}+\bfs\sigma_{2*};\tau_1+\tau_2,\tau_{1*}+\tau_{2*})}.$$

\ms The following version of BPZ-Cardy equations holds for rooted multi-vertex fields $\OO \equiv\OO^{(\bfs\sigma,\bfs\sigma_*;\tau,\tau_*)}$ with the neutrality condition.
Like BPZ-Cardy equations for the tensor product of fields in the OPE family of $\Phi_{(b)},$ we derive the following proposition from Ward's equations for rooted multi-vertex fields and the level two degeneracy equations for $\oneleg.$

\ms \begin{prop} \label{BPZ4O}
If $2a(a+b)=1,$ then in the identity chart of $\D$ we have
\begin{equation}\label{eq: BPZ4O}
-\frac{1}{a^2}\partial_\theta^2 \wh\E_\xi \OO = \wh\E_\xi[\LL_{v_\xi}\OO]+(\wh h_q + \wh h_{q*})\wh\E_\xi[\OO],\qquad v_\xi(z):=z\frac{\xi+z}{\xi-z},
\end{equation}
where $[\wh h_q,\wh h_{q*}]$ are the conformal dimensions of $\wh\OO$ at $q,$ and  $\pa_\theta$ is the operator of differentiation with respect to the real variable $\theta.$
\end{prop}

The Lie derivative operators $\LL_{v_\xi}$ do not apply to the marked interior point $q,$ i.e., $\LL_{v_\xi} = \LL({v_\xi},\D\sm\{0\}).$ 

\begin{proof}
For $\xi\in\pa\D,\zeta\in\D, $ denote
$$R_\xi\equiv R_\xi(z_1,\cdots, z_n)=\wh{\E}_\xi \OO, \quad R_\zeta\equiv R(\zeta;z_1,\cdots, z_n)=\E [\zeta^h \oneleg(\zeta)\star\OO],$$
where $h =h(\oneleg)= a^2/2 -ab.$
Since $\oneleg$ is a holomorphic $[h,0]$-differential ($\LL_v^-\oneleg=0$), it follows from Proposition~\ref{Ward4VX} that 
\begin{align*}
\E\,\oneleg(\zeta)\star(\LL_{v_\zeta}^+\OO+\LL_{v_{\zeta^*}}^-\OO )
&=2\zeta^2\,\E[(L_{-2}\oneleg)(\zeta)\star\OO] +3\zeta\,\E[(L_{-1}\oneleg)(\zeta)\star\OO]\\ &+(h(\oneleg)-H_q^\eff(\oneleg\star\OO))\,\E\,\oneleg(\zeta)\star\OO.
\end{align*}
The effective dimensions of $\oneleg\star\OO$ and $\oneleg$ are related as 
$$H_q^\eff(\oneleg\star\OO) = H_q^\eff(\oneleg) + \wh h_q + \wh h_{q*}.$$
As in the proof of Proposition~\ref{BPZ4X} , it follows from the fusion rule ($h-H_q^\eff(\oneleg) = h^2/a^2$), $L_{-1}\oneleg = \pa\oneleg$ and the level two degeneracy equations for $\oneleg$ ($2a^2L_{-2}\oneleg = \pa^2\oneleg$) that 
\begin{align*}
\E\,\oneleg(\zeta)\star(\LL^+_{v_\zeta}\OO +\LL^-_{v_{\zeta^*}}\OO)&= \frac{1}{a^2}\zeta^{2}\pa_\zeta^2(\zeta^{-h}R_\zeta)+3\zeta\pa_\zeta(\zeta^{-h} R_\zeta)\\&+\Big(\frac{h^2}{a^2}-\wh h_q - \wh h_{q*}\Big)\,\zeta^{-h}R_\zeta.
\end{align*}
It simplifies (by the numerology $2a(a+b)=1$)
$$\E\,\zeta^{h}\oneleg(\zeta)\star(\LL^+_{v_\zeta}\OO +\LL^-_{v_{\zeta^*}}\OO)= \frac1{a^2}\big(\zeta^2\pa_\zeta^2R_\zeta + \zeta\pa_\zeta R_\zeta\big) - (\wh h_q + \wh h_{q*})R_\zeta.$$
Sending $\zeta$ to $\xi,$ we get BPZ-Cardy equations~\eqref{eq: BPZ4O}.
\end{proof}

\ms We now prove that the non-random fields $\E\,\wh\OO^{(\bfs\sigma,\bfs\sigma_*;\tau,\tau_*)}$ (with the neutrality condition) are radial SLE martingale-observables.

\ms \begin{proof}[Proof of Theorem~\ref{main O}]
The process $M_t =M_{D_t,\gamma_t,q}$ is represented by 
\begin{equation} \label{eq: M}
M_t = m(\xi_t,t), \qquad m(\xi,t) =\big(R_\xi\,\|\,g_t^{-1}\big),
\end{equation}
where $g_t$ is the SLE conformal map, $g_t(\gamma_t)=\xi_t = e^{i\theta_t},$ and $\theta_t=\sqrt\kappa B_t.$
Applying It\^o's formula to the process $m(\xi_t,t)$ as in the proof of Theorem~\ref{MO}, 
$dM_t$ has the drift term 
\begin{equation*} 
-\frac\kappa2\Big((\xi\pa_\xi)^2\Big|_{\xi=\xi_t}~m(\xi,t)\Big)\,dt+\big(\LL_{v_{\xi_t}}R_{\xi_t}\,\|\,g_{t}^{-1}\big)\, dt + (\wh h_q+\wh h_{q*})M_t\,dt,
\end{equation*}
where the Lie derivative operators $\LL_{v_{\xi_t}}$ do not apply to the origin, i.e., $\LL_{v_{\xi_t}} = \LL({v_{\xi_t}},\D\sm\{0\}).$ 
The extra term $(\wh h_q+\wh h_{q*})M_t\,dt$ comes from the covariance structure of $M$ at $q$ and the fact that $g_t'(0) = e^t.$ (The chart in the representation \eqref{eq: M} is not $w_t^{-1}$ but $g_t^{-1}.$)
By BPZ-Cardy equations~\eqref{eq: BPZ4O}, the drift term of $dM_t$ vanishes.
\end{proof} 

\ms Multi-vertex observables are not covering all solutions of It\^o's equation. 
A collection of primary observables can be further expanded by the method of screening. 

\ms
\def\cprime{$'$}
\providecommand{\bysame}{\leavevmode\hbox to3em{\hrulefill}\thinspace}
\providecommand{\MR}{\relax\ifhmode\unskip\space\fi MR }
\providecommand{\MRhref}[2]{%
  \href{http://www.ams.org/mathscinet-getitem?mr=#1}{#2}
}
\providecommand{\href}[2]{#2}


\end{document}